\documentclass[11pt]{amsart}

\usepackage[english]{babel}
\usepackage{amsmath}
\usepackage{amssymb}
\usepackage{amsthm}
\usepackage{latexsym}
\usepackage{mathtools}
\usepackage{thmtools}
\usepackage{amsfonts}
\usepackage{mathrsfs}
\usepackage{textcomp}
\usepackage[T1]{fontenc}
\usepackage{graphicx}
\usepackage{setspace}
\usepackage{nicefrac}
\usepackage{indentfirst}
\usepackage{enumerate}
\usepackage{wasysym}
\usepackage{upgreek}
\usepackage{booktabs}
\usepackage{etoolbox}
\usepackage{mathdots}
\usepackage{wrapfig}
\usepackage{floatflt}
\usepackage{tensor} 
\usepackage{parskip}

\usepackage[top=3cm,bottom=3cm,left=3.5cm,right=3.5cm]{geometry}

\newtheorem{thm}{Theorem}[section]
\newtheorem{prop}[thm]{Proposition}
\newtheorem{lemma}[thm]{Lemma}
\newtheorem{corol}[thm]{Corollary}

\theoremstyle{definition}
\newtheorem{defin}[thm]{Definition}
\newtheorem{remark}[thm]{Remark}

\newtheorem{example}[thm]{Example}

\usepackage{breqn}

\renewcommand{\Im}{\mathrm{Im} \;}  

\newcommand{\uR}{\underline{\mathrm{R}}}

\newcommand{\rank}{\mathrm{rank}}

\newcommand{\aR}{\mathrm{R}^{\kern-1pt\scriptscriptstyle\otimes}}
\newcommand{\uaR}{\underline{\mathrm{R}}^{\kern-1pt\scriptscriptstyle\otimes}}

\usepackage[normalem]{ulem}

\newcommand{\xto}[1]{\xrightarrow{\phantom{a}{#1}{\phantom{a}}}}

\newcommand{\vvirg}{ , \dots , }
\newcommand{\ootimes}{ \otimes \cdots \otimes }

\newcommand{\ttimes}{ \times \cdots \times }

\newcommand{\dotitem}{ \item[$\cdot$]}

\newcommand{\rmR}{\mathrm{R}}
\newcommand{\bbC}{\mathbb{C}}
\newcommand{\bbP}{\mathbb{P}}
\newcommand{\bbZ}{\mathbb{Z}}
\newcommand{\bbN}{\mathbb{N}}
\newcommand{\frakS}{\mathfrak{S}}
\newcommand{\frakm}{\mathfrak{m}}
\newcommand{\scrS}{\mathscr{S}}
\newcommand{\bfx}{\mathbf{x}}

\newcommand{\textsum}{{\textstyle \sum}}

\newcommand{\textfrac}[2]{{\textstyle \frac{#1}{#2}}}

\renewcommand{\phi}{\varphi}
\newcommand{\eps}{\varepsilon}
\renewcommand{\theta}{\vartheta}

\renewcommand{\hat}[1]{\widehat{#1}}
\renewcommand{\bar}[1]{\overline{#1}}

\DeclareMathOperator{\Hom}{Hom}
\DeclareMathOperator{\mult}{mult}
\DeclareMathOperator{\codim}{codim}

\newcommand{\dashto}{\dashrightarrow}

\newcommand{\id}{\mathrm{id}}

\title{Border rank is not multiplicative under the tensor product}

\author[M. Christandl]{Matthias Christandl}
\author[F. Gesmundo]{Fulvio Gesmundo}
\author[A. K. Jensen]{Asger Kj{\ae}rulff Jensen}

\address[M. Christandl, F. Gesmundo, A. K. Jensen]{QMATH - Dept. of Math. Sciences, Univ. of Copenhagen, Universitetsparken 5, 2100, Copenhagen, Denmark}

\email[Christandl]{christandl@math.ku.dk}
\email[Gesmundo]{fulges@math.ku.dk}
\email[Jensen]{akj@math.ku.dk}

\begin{document}
\begin{abstract}
 It has recently been shown that the tensor rank can be strictly submultiplicative under the tensor product, where the tensor product of two tensors is a tensor whose order is the sum of the orders of the
two factors. The necessary upper bounds were obtained with help of border rank. It was left open whether border rank itself can be strictly submultiplicative. We answer this question in the affirmative. In order to do so, we construct lines in projective space along which the border rank drops multiple times and use this result in conjunction with a previous construction for a tensor rank drop. Our results also imply strict submultiplicativity for cactus rank and border cactus rank.
\end{abstract}

 \maketitle
 
 \section{Introduction}\label{section: intro}
 
We work over the complex numbers and we will point out when our results apply in higher generality. Given vector spaces $V_1 \vvirg V_k$ and a tensor $T \in V_1 \ootimes V_k$, the \emph{tensor rank} of $T$ is 
 \[
 \rmR(T) = \min \left\{r: T = \textsum_{i = 1}^r v_1^{(i)} \ootimes v_{k}^{(i)}, \text{ for some } v_j^{(i)} \in V_j\right\}.
\]
If $k=2$, then $\rmR(T) = \rank(T)$ where $T$ is regarded as a linear map $T: V_1^{*} \to V_2$; in this sense, tensor rank is a generalization of matrix rank.

The \emph{tensor border rank} (\emph{border rank}, for short) of $T$ is 
\[
 \uR(T) = \min \left\{ r : T = \lim_{\eps \to 0} T_\eps \text{ where, for every $\eps$, $\rmR(T_\eps) = r$} \right\}
\]
and the limit is taken in the Euclidean topology of $V_1 \ootimes V_k$. Clearly $ \uR(T) \leq \rmR(T)$ and there are many examples where the inequality is strict.

It is straightforward to verify that rank and border rank are submultiplicative under the tensor product: if $T_1$ and $T_2$ are tensors of order $k_1,k_2$ respectively, then $T_1 \otimes T_2$ is a tensor of order $k_1 + k_2$ satisfying $\rmR(T_1 \otimes T_2) \leq \rmR(T_1) \rmR(T_2)$ and $\uR(T_1 \otimes T_2) \leq \uR(T_1) \uR(T_2)$. Recently, \cite{ChrJenZui:NonMultiplicativityTensorRank} answered a question posed in \cite{Draisma:MultAlgebraLectureNotes} and provided the first example showing that submultiplicativity of rank can be strict, namely $\rmR(T_1 \otimes T_2) < \rmR(T_1) \rmR(T_2)$.

The analogous question for border rank, namely whether border rank can be strictly multiplicative under tensor product, remained open and we answer it in this paper. Specifically, we provide an example of a tensor $T$ such that $\uR(T) = 5$ and $\uR(T \otimes T) \leq 24$. We obtain
\begin{thm}\label{thm: not multiplicative}
 Border rank is not multiplicative under the tensor product.
\end{thm}

In the spirit of Strassen's asymptotic rank conjecture (see \cite{Str:DegComplexityBilMaps}), the existence of examples that verify strict submultiplicativity of rank and border rank motivates the definition of a \emph{tensor asymptotic rank} of a tensor:
\begin{equation}\label{eqn: asymptotic rank def}
\aR (T) = \lim_{k \to \infty} [\rmR(T^{\otimes k})]^{1/k}. 
\end{equation}
This notion is different from the asymptotic rank $\uwave{\rmR}$ defined in \cite{Str:AsySpectrumTensorsExpMatMult}. In that case, the object of study is the asymptotic rank under Kronecker product (or flattened tensor product, denoted by $\boxtimes$) where the $k$-th tensor power of $T \in V_1 \ootimes V_\ell$ is regarded as a tensor of order $\ell$ in $(V_1^{\otimes k}) \ootimes (V_\ell^{\otimes k})$, denoted $T^{\boxtimes k}$. In \eqref{eqn: asymptotic rank def}, $T^{\otimes k}$ is regarded a tensor of order $\ell k$. In particular, the word \emph{tensor} in ``tensor asymptotic rank'' refers to the fact that we are considering tensor powers. The tensors $T^{\otimes k}$ and $T^{\boxtimes k}$ are formally the same object but there is a difference in the choice of rank $1$ tensors that are used to compute their ranks: in both cases, the rank is computed (after reordering) with respect to rank $1$ tensors of the form $Z_1 \ootimes Z_\ell$ with $Z_j \in V_j^{\otimes k}$ with the important difference that in the case of $T^{\boxtimes k}$ no condition on $Z_1 \vvirg Z_\ell$ is given, whereas in the case of $T^{\otimes k}$ we require that every $Z_j$ is a rank $1$ tensor in $V_j \ootimes V_j$, namely $Z_j = v_{j,1} \ootimes v_{j,k}$. In particular, it is clear that $\rmR(T^{\boxtimes k}) \leq \rmR(T^{\otimes k})$ and therefore $\uwave{\rmR}(T) \leq \aR (T)$.

A consequence of \cite{ChrJenZui:NonMultiplicativityTensorRank} is that the inequality $\aR(T) \leq \rmR(T)$ can be strict. Moreover, Theorem 8 in \cite{ChrJenZui:NonMultiplicativityTensorRank} guarantees that $\aR(T) = \lim_{k \to \infty} [\uR(T^{\otimes k})]^{1/k}$, namely the definition of the tensor asymptotic rank does not depend on considering rank or border rank and in particular $\aR(T) \leq \uR(T)$; Theorem \ref{thm: not multiplicative} shows that the inequality $\aR(T) \leq \uR(T)$ can be strict as well. 

The possible gap between $\uR(T)$ and $\uwave{\rmR}(T)$ is due to two different phenomena: Theorem \ref{thm: not multiplicative} shows that in general there might a gap due to tensoring together several copies of a tensor, namely $\uR(T^{\otimes k})^{1/k} < \uR(T)$; on the other hand, there might be an additional gap due to passing from tensor product to Kronecker product, namely $\uR(T^{\boxtimes k}) < \uR(T^{\otimes k})$ (this follows from flattening lower bounds multiplicativity, as in \cite{ChrJenZui:NonMultiplicativityTensorRank}). In summary, we have the sequence of inequalities
\[
 \uwave{\rmR}(T) \leq \aR(T) \leq \uR(T) \leq \rmR(T),
\]
and each of these inequalities can be strict in some cases.

It is natural to ask whether asymptotic rank itself is multiplicative under the tensor product (this problem was posed in \cite{ChrJenZui:NonMultiplicativityTensorRank}), i.e. whether 
\[
\uwave{\rmR}(T \otimes T') = \uwave{\rmR} (T) \uwave{\rmR}(T').  
\]
We leave this question open, and only point out that an answer can in principle be found with help of the asymptotic spectrum of tensors, which is able to characterize the asymptotic rank \cite{Str:AsySpectrumTensors}. The best known lower bound on $\uwave{\rmR}$ is the so-called maximal local dimension and it is consistent with current knowledge that this bound is sharp \cite{ChrVraZui:UniversalPtsAsySpecTensors}. In the case of tight tensors (in the sense of \cite{Str:DegComplexityBilMaps}), this is a conjecture of Strassen. If this was true, multiplicativity would be immediate. Note that the asymptotic rank is not multiplicative under the Kronecker tensor product. 

However, multiplicativity of flattening lower bounds \cite{ChrJenZui:NonMultiplicativityTensorRank} provides nontrivial lower bounds on the tensor asymptotic rank whenever there is a flattening map providing lower bounds higher than the local dimension of the tensor. In particular, in Proposition \ref{prop: multiplicativity for m plus 1}, we provide an explicit example of a tensor in $\bbC^m \otimes \bbC^m \otimes \bbC^m$ with $\aR(T) = m+1$.

The definition of the tensor asymptotic rank is motivated by the importance of the asymptotic rank $\uwave{\rmR}$ in the study of the complexity of tensors. In algebraic complexity theory, there is great interest in understanding the complexity of performing the multiplication of two matrices. This complexity is asymptotically controlled by the asymptotic rank of the matrix multiplication tensor $M_{\langle n \rangle} \in \bbC^{n^2} \otimes \bbC^{n^2} \otimes \bbC^{n^2}$, that is the bilinear map sending a pair of matrices to their product. It is known that $\uwave{\rmR}(M_{\langle n \rangle}) = n^\omega$ for a constant $\omega$, the so-called exponent of matrix multiplication; \cite{Bini:RelationsExactApproxBilAlg} showed that $\uwave{\rmR}(M_{\langle n \rangle}) < \uR(M_{\langle n \rangle})$ for every $n$ and it is conjectured that $\omega = 2$ in the computer science community; the current state of the art is $2 \leq \omega < 2.37287$ \cite{Stot:ComplexityMaMu,Wil:FasterThanCW,LeGallPowersTensorsFastMaMu}. More generally, Strassen's asymptotic rank conjecture states that $\uwave{\rmR}(T) = m$ for every \emph{tight, concise} tensor in $\bbC^m \otimes \bbC^m \otimes \bbC^m$; we refer to \cite{Str:DegComplexityBilMaps,ChrVraZui:UniversalPtsAsySpecTensors,ConGesLanVenWan:GeometryStrassenAsyRankConj} for the details on this topic. 

We expect that a better understanding of the tensor asymptotic rank will lead to a better understanding of the gap in $\uwave{\rmR}(M_{\langle n \rangle}) < \uR(M_{\langle n \rangle})$, and more generally in other cases of interest.

Another field where submultiplicativity properties of tensor rank and border rank play a role is quantum information theory, where the state of a quantum system is represented as a vector in a Hilbert space and the state of a composite system is an element of the tensor product of the Hilbert spaces corresponding to its constituents. In this theory, tensor rank is a natural measure of the entanglement of the quantum system \cite{DurVidCir:TrheeCubitsEntTwoIneqWays,EisBri:SchmidtMeasureToolQuantifyingMultipEntangl,BernCaru:AlgGeomToolsEntanglement}. In particular, strict submultiplicativity properties of tensor rank reflect the fact that in a quantum system formed by multiple independent constituents, the entanglement is not simply ``the sum'' of the entanglement of the different parts (see also \cite{YuChiGuoDu:TensorRankTripartiteW,BalBerChrGes:PartiallySymRkW}).

A similar unexpected consequence of strict submultiplicativity can be found in the geometric interpretation of the quantum broadcast model in communication complexity (see, e.g., \cite{BuhChrZui:NondetQCC_CyclicEqGameIMM}). Here, the tensor encodes a Boolean function that distant parties have to jointly compute using as little quantum communication as possible (measured in terms of entanglement). Submultiplicativity properties of rank and border rank show that if two groups of parties play two independent games of this type, then a joint strategy can be advantageous compared to two independent optimal strategies. In \cite{BuhChrZui:NondetQCC_CyclicEqGameIMM}, the border rank version of this protocol is discussed as well; the results of this paper show that a joint strategy can be advantageous even in the border rank setting.

From the geometric point of view, border rank is directly related to the study of secant varieties of the Segre product of projective spaces, a topic that has been studied in the algebraic geometry community for over a century (see, e.g., \cite{Lan:TensorBook}). Indeed, one can define rank and border rank with respect to any algebraic variety (see Section \ref{section: preliminaries}) and address the submultiplicativity problem in much higher generality. It is interesting to observe that one of the arguments used in \cite{ChrJenZui:NonMultiplicativityTensorRank} to prove tensor rank strict submultiplicativity already contains the seed of the argument that can be applied in complete generality. For this reason, we will study the problem from the more general perspective of secant varieties and $X$-rank (in the sense of Section \ref{section: preliminaries}). In this context, one can easily obtain examples of $X$-rank and $X$-border rank submultiplicativity when $X$ is a set of distinct points (\cite{Sko:personalComm} -- we briefly show this example in Remark \ref{remark: drop with distinct points}).

The paper is structured as follows. In Section \ref{section: preliminaries}, first we briefly retrace that argument of \cite{ChrJenZui:NonMultiplicativityTensorRank} focusing on the elements that can be transferred to the border rank setting, then we introduce some basic notions from algebraic geometry that will be useful in the rest of the paper. In Section \ref{section: P2xP2xP2 basic}, we present an example of border rank strict submultiplicativity, which implies Theorem \ref{thm: not multiplicative}. In Section \ref{section: secant multidrop}, we present the more general geometric argument that led to the example presented in Section \ref{section: P2xP2xP2 basic} and we prove some results providing sufficient conditions under which the argument can be applied. In Section \ref{section: two hypersurface examples}, we discuss several examples where these sufficient conditions are satisfied, giving additional examples of border rank strict submultiplicativity, some of which involve the geometry of algebraic curves.  In Section \ref{section:AsymptoticRank}, we define a more general notion of asymptotic rank, that extends the definition given above to every algebraic variety and we prove some results about this quantity. Finally, Section \ref{section: other varieties} is dedicated to additional examples and numerical results.

\subsection*{Acknowledgments} We thank L. Chiantini for suggesting to investigate the example of elliptic curves, T. Fisher for the references on this topic and K. Kordek for helpful discussions. We thank J. Skowera and J. Zuiddam for discussions on $X$-rank and $X$-border rank multiplicativity. We acknowledge financial support from the European Research Council (ERC Grant Agreement no. 337603), the Danish Council for Independent Research (Sapere Aude), and VILLUM FONDEN via the QMATH Centre of Excellence (Grant no. 10059).
  
 \section{Preliminaries}\label{section: preliminaries}
 
 In this section we introduce some of the tools that we will use in the paper and we present the known results about rank strict submultiplicativity that led us to the construction to prove border rank strict submultiplicativity.

\subsection{Rank submultiplicativity} \label{subsec: rank submult}Proposition 13 in \cite{ChrJenZui:NonMultiplicativityTensorRank} provides a minimal example of strict submultiplicativity of tensor rank. This construction is an optimized version of the more general interpolation argument of Theorem 8 in \cite{ChrJenZui:NonMultiplicativityTensorRank}. Here, we present that example pointing out the key property that allows us to use a similar argument in the case of border rank.
  
  Let $A,B,C$ be three $2$-dimensional vector spaces; let $a_1,a_2$ be a basis of $A$, $b_1,b_2$ a basis of $B$ and $c_1,c_2$ a basis of $C$. Let $W = a_2 \otimes b_1 \otimes c_1 + a_1 \otimes b_2 \otimes c_1 + a_1 \otimes b_1 \otimes c_2 \in A \otimes B \otimes C$. It is known that $\rmR(W) = 3$ and $\rmR(W + \eps a_2 \otimes b_2 \otimes c_2) = 2$ for every $\eps \neq 0$.
 
We obtain the following expression for $W^{\otimes 2}$:
\begin{equation}\label{eqn: construction W drop}
\begin{aligned}
 W \otimes W = (W - a_2 \otimes b_2 \otimes c_2) ^ {\otimes 2} &+ (W - \textfrac{1}{2} a_2 \otimes b_2 \otimes c_2) \otimes (a_2 \otimes b_2 \otimes c_2) \\ &+ (a_2 \otimes b_2 \otimes c_2) \otimes (W - \textfrac{1}{2} a_2 \otimes b_2 \otimes c_2).
\end{aligned}
\end{equation}
providing $\rmR(W \otimes W) \leq 2 \cdot 2 + 2 \cdot 1 + 1 \cdot 2 = 8 < 9 = 3 \cdot 3$, which gives an example of strict submultiplicativity of tensor rank. Following this proof, \cite{ChFri:TensorRankTensorProductTwoW} proved that $\rmR(W \otimes W) \geq 8$, and thus $\rmR(W \otimes W) = 8$.

We stress that the key elements that are used in this construction are that $\rmR(W + \eps a_2 \otimes b_2 \otimes c_2) = 2$ for $\eps = \frac{1}{2}$ and $\eps =1$ and that $\rmR(a_2 \otimes b_2 \otimes c_2) = 1$. We will generalize this construction in the setting of border rank as follows: we will determine two tensors $T,Z$ such that $\uR( T - 2Z) = \uR(T - Z) = \uR(T) -1$ and $\uR(Z) = 1$. Then, the analog of expression \eqref{eqn: construction W drop} will show that $T$ verifies strict submultiplicativity of border rank.

 \subsection{Geometry} We denote with $\bbP^N$ the projective space of lines in $\bbC^{N+1}$. The term \emph{variety} always refers to a projective or affine algebraic variety; moreover, varieties are always considered to be nondegenerate, namely not contained in a hyperplane, unless stated otherwise. If $X \subseteq \bbP^N$ is a variety, we denote by $\hat{X}$ the affine cone over $X$ in $\bbC^{N+1}$ and by $\langle X \rangle$ the projective span of the variety $X$. A variety $X$ is called irreducible if it is not the union of two proper subvarieties.
 
 We refer to \cite{Harris:AlgGeo} for the notions of dimension (Lecture 11), degree (Lecture 18), tangent space (Lecture 14) and tangent cone (Lecture 20) of an algebraic variety. If $X \subseteq \bbP V$ is a variety, we denote by $I(X)$ the homogeneous ideal of $X$, which is a homogeneous ideal in the symmetric algebra $Sym(V^*)$. We denote by $I_d(X)$ the homogeneous component of degree $d$, that is a linear subspace of $S^d V^*$. A variety $X \subseteq \bbP^N$ of dimension $N-1$ is called hypersurface; in this case $I(X)$ is a principal ideal and the degree of $X$ is equal to the degree of a generator of $I(X)$.
 
 Let $X \subseteq \bbP^N$ and let $p \in \bbP^N$. The $X$-rank of $p$ is 
 \[
  \rmR_{X}(p) = \min \{ r: p \in \langle z_1 \vvirg z_r \rangle \text{ for some $z_1 \vvirg z_r \in X$}\}
 \]
and we write $\sigma_r^\circ(X) = \{ p \in \bbP^N : \rmR_{X}(p) \leq r\}$, the set of all points having $X$-rank at most $r$. The $r$-th secant variety of $X$ is $\sigma_r(X) = \bar{\sigma_r^\circ(X)}$, where the overline denotes the closure in the Zariski topology. The $X$-border rank of $p$ is
\[
 \uR_X(p) = \min \{ r: p \in \sigma_r(X) \}.
\]
It is a fact that secant varieties of irreducible varieties are irreducible (see, e.g., \cite{Harris:AlgGeo}, Lecture 8).

Now, fix vector spaces $V_1 \vvirg V_k$ and consider the Segre embedding  $Seg: \bbP V_1 \ttimes \bbP V_k \to \bbP (V_1 \ootimes V_k)$ defined by $Seg([v_1]\vvirg [v_k]) = [v_1 \ootimes v_k]$ where the brackets $[ \cdot ]$ denote the class of a vector in the corresponding projective space. Then $Seg(\bbP V_1 \ttimes \bbP V_k)$ is a variety in $\bbP (V_1 \ootimes V_k)$. A tensor $T$ has tensor rank $r$ if and only if the point $[T]$ has $(Seg(\bbP V_1 \ttimes \bbP V_k))$-rank $r$. The same is true for border rank because if $X$ is the Segre variety, then the closure of $\sigma_r^\circ(X)$ in the Zariski topology is the same as its closure in the Euclidean topology (this is a consequence of \cite{Mum:ComplProjVars}, Thm. 2.33).

\begin{remark}[\cite{Sko:personalComm}] \label{remark: drop with distinct points}
 We observe that when $X$ is a set of distinct points, an example of $X$-border rank submultiplicativity is immediate. Since in this case $X$-rank and $X$-border rank coincide (secant varieties are just arrangements of linear spaces), we only need to determine an example of $X$-rank submultiplicativity. Let $X \subseteq \bbP^1$ be a set of three general points; without loss of generality assume $X = \{ (1,0),(0,1),(1,1) \}$. We determine a point $p = q \otimes q \in Seg(\bbP^1 \times \bbP^1) \subseteq \bbP^3$ such that $\rmR_{X \times X} (p) = 3 < 4 = \rmR_X(q) ^2$ (and the same for border rank). Notice that for every $q \in \bbP^1$, $q \notin X$ we have $\uR_X (q) = \rmR_X (q) = 2$. Let $q = (1,-1)$ and $p = q \otimes q$. Then 
 \[
 p = q \otimes q = (1,1) \otimes (1,1) - 2 \cdot (1,0) \otimes (1,0) - 2 \cdot (0,1) \otimes (0,1),
 \]
proving $\rmR_{X \times X}(p) = \uR_{X \times X}(p) \leq 3$. Equality easily follows. 
\end{remark}

 \subsection{Varieties and groups}\label{subsection: gvarieties} Let $G$ be an algebraic group acting linearly on a vector space $V$, via a representation $G \to GL(V)$. Then the action passes to the projective space $\bbP V$, via the projection $GL(V) \to PGL(V)$; if $p \in \bbP V$, we denote by $G \cdot p$ the orbit of $p$ under the action of $G$. Let $X \subseteq \bbP V$ be an algebraic variety. We say that $X$ is a $G$-homogeneous space if $X = G\cdot p$ for some (and indeed for every) $p \in X$. We say that $X$ is quasi-$G$-homogeneous, or a $G$-orbit-closure, if $X = \bar{G \cdot p}$ for some (and indeed for a generic) $p \in X$, where the overline denotes the closure in the Zariski (or equivalently Euclidean) topology. We say that $X$ is a $G$-variety if it is closed under the action of $G$, namely $G\cdot p \subseteq X$ for every $p \in X$. It is immediate that if $X$ is a $G$-variety, then all its secant varieties are $G$-varieties.
 
The action of a group $G$ on $V$ defines via pullback an action on the symmetric algebra $Sym ( V^*)$: indeed, if $g \in G$ and $f \in Sym (V^*)$ is a polynomial on $V$, then $g \cdot f = f \circ g^{-1}$ defines a degree preserving linear action on $Sym (V^*)$. In particular, the homogeneous components $S^d V^*$ are $G$-representations. If $X$ is a $G$-variety, then $I_d(X)$ is a $G$-submodule of $S^d V^*$.

If $S \subseteq \bbP V$ is a hypersurface of degree $d$ that is also a $G$-variety, then $I_d(X)$ has to be a $1$-dimensional representation of $G$.

 The Segre variety defined in the previous section is an example of a homogeneous variety: $Seg( \bbP V_1 \ttimes \bbP V_k) \subseteq \bbP( V_1 \ootimes V_k)$ is the orbit of $[v_1 \ootimes v_k]$ under the action of $SL(V_1) \ttimes SL(V_k)$. Other examples of homogeneous varieties are Veronese varieties, Segre-Veronese varieties, Grassmannians, flag varieties and smooth quadrics. We refer to \cite[\S 6.9]{Lan:TensorBook} for additional information.

 \subsection{Flattenings}\label{subsec: flattenings} Let $V$ be a vector space. A \emph{flattening} of $V$ is a linear map $Flat_{E,F} : V \to \Hom (E,F)$ where $E,F$ are two vector spaces. Flattening maps are a classical approach to determining equations for secant varieties, although recently strong limitations on this
technique have been proved (see \cite{Galazka:VectorBundlesGiveEqnsForCactus} and \cite{EfrGarOliWig:BarriersRankMethods}). If $X \subseteq \bbP V$ is an algebraic variety, and $p \in \bbP V$, then the rank of $Flat_{E,F}(p)$ provides a lower bound on the $X$-border rank of $p$ as follows: let $r_0 = \max \{\rank ( Flat_{E,F}(z)) : z \in X \}$; then (see, e.g., Proposition 4.1.1 in \cite{LanOtt:EqnsSecantVarsVeroneseandOthers} and Lemma 18 in \cite{ChrJenZui:NonMultiplicativityTensorRank})
 \begin{equation}\label{eqn: flattening lower bound}
  \uR_X(p) \geq \frac{1}{r_0} \rank( Flat_{E,F}(p)).
 \end{equation}
In particular, minors of size $r \cdot r_0 +1$ of $Flat_{E,F}(p)$ give equations for $\sigma_r(X)$.

In the tensor setting, an element $T \in V_1 \ootimes V_k$ defines naturally a linear map for every set of indices $J \subseteq \{ 1 \vvirg k\}$, by $\bigotimes_{j \in J} V_j^* \to \bigotimes_{\ell \notin J} V_\ell$, via the natural contraction of $T$ on the factors of $J$. This defines a flattening map (which is indeed an isomorphism) $V_1 \ootimes V_k \to \Hom (\bigotimes_{j \in J} V_j^*, \bigotimes_{\ell \notin J} V_\ell)$, that is often called standard flattening. It is easy to show that $T$ is a rank $1$ tensor if and only if every standard flattening has rank $1$ on $T$. In particular, for this type of flattening, the value $r_0$ defined above is $1$.

Proposition 20 in \cite{ChrJenZui:NonMultiplicativityTensorRank} proves that flattening lower bounds are multiplicative in the following sense. For $i=1,2$, let $X_i \subseteq \bbP V_i$ be a variety and let $Flat_{E_i,F_i}$ be a flattening of $V_i$, with corresponding value $r_{i,0} =  \max \{\rank ( Flat_{E_i,F_i}(z_i)) : I \in X_i \}$. Let $p_i \in \bbP V_i$. Then
\begin{equation}\label{eqn: multiplicative flattening lower bounds}
  \uR_{X_1 \times X_2}(p_1 \otimes p_2) \geq \frac{1}{r_{0,1}} \rank( Flat_{E_1,F_1}(p_1)) \cdot \frac{1}{r_{0,2}} \rank( Flat_{E_2,F_2}(p_2)).
\end{equation}
We point out that in general the flattening lower bound in \eqref{eqn: flattening lower bound} is not an integer, and since $\uR_X(p) $ is an integer one obtains $R_X(p) \geq \lceil \frac{1}{r_0} \rank( Flat_{E,F}(p)) \rceil$; however, the multiplicativity result on the lower bounds only applies to the actual flattening bound, rather than to their ceilings. In particular, from \eqref{eqn: multiplicative flattening lower bounds}, one obtains $\uR_{X_1 \times X_2}(p_1 \otimes p_2) \geq \lceil \frac{1}{r_{0,1}} \rank( Flat_{E_1,F_1}(p_1)) \cdot \frac{1}{r_{0,2}} \rank( Flat_{E_2,F_2}(p_2)) \rceil$ and not $\uR_{X_1 \times X_2}(p_1 \otimes p_2) \geq \lceil \frac{1}{r_{0,1}} \rank( Flat_{E_1,F_1}(p_1)) \rceil \cdot \lceil \frac{1}{r_{0,2}} \rank( Flat_{E_2,F_2}(p_2)) \rceil$. This fact is one of the main insights that allowed us to find the example of strict submultiplicativity on tensor rank that we present in the next section.

\section{An example of strict submultiplicativity for border rank}\label{section: P2xP2xP2 basic}

In this section, we present an example of tensor $T$ with $\uR(T^{\otimes 2}) < \uR(T)^2$, giving an answer to the problem of strict submultiplicativity of border rank. The geometric description of this example is classical and it relies on the properties of the variety of tensors of border rank $4$ in $\bbC^3 \otimes \bbC^3 \otimes \bbC^3$. In Section \ref{section: two hypersurface examples}, we will give additional information on this variety, relating it with the general geometric framework that we present in Section \ref{section: secant multidrop}.

Let $A,B,C$ be three $3$-dimensional vector spaces and let $X = Seg(\bbP A \times \bbP B \times \bbP C)$ be the Segre variety in $\bbP(A\otimes B \otimes C)$ so that the affine cone $\hat{X}$ is the variety of rank $1$ tensors in $A \otimes B \otimes C$. Fix bases $a_1, a_2, a_3 \in A$, $b_1, b_2, b_3 \in B$ and $c_1, c_2, c_3 \in C$ with dual bases $\alpha_1, \alpha_2,\alpha_3 \in A^*$, $\beta_1, \beta_2,\beta_3 \in B^*$ and $\gamma_1, \gamma_2,\gamma_3 \in C^*$.

\begin{prop}\label{prop: explicit drop for C3 ot C3 ot C3}
Define 
\begin{align*}
T := &a_1 \otimes b_1 \otimes c_1 +  a_2 \otimes b_2 \otimes c_2 +a_3 \otimes b_3 \otimes c_3 +  \\ + &(\textsum_1^3 a_i)\otimes (\textsum_1^3 b_i) \otimes (\textsum_1^3 c_i) + 2 (a_1 + a_2)\otimes (b_1+b_3)\otimes (c_2+c_3).
\end{align*}
Then $\uR(T) = 5$ and $\uR(T^{\otimes 2} ) \leq 24 < 5^2$.
\end{prop}
\begin{proof}
The upper bound $\uR(T) \leq 5$ is clear from the expression of $T$. The lower bound is provided by the flattening map $ T_B^{\wedge A} : A \otimes B^* \to \Lambda^2 A \otimes C$ defined as the composition
\begin{equation}\label{eqn: strassen's flattening}
A \otimes B^* \xto {\id_A \otimes T_B} A \otimes A \otimes C  \xto{\pi_\Lambda \otimes \id_C} \Lambda^2 A \otimes C
\end{equation}
where $T_B : B^* \to A \otimes C$ is the standard tensor contraction and $\pi_\Lambda : A \otimes A \to \Lambda^2 A$ is the projection onto the skew-symmetric component.

Fixing bases $\{ a_i \otimes \beta_j \}$ in the domain and $\{a_i \wedge a_j \otimes c_k\}$ in the codomain (ordered lexicographically), one can see that the matrix associated to the linear map $T_B^{\wedge A}$ is
\[
T_B^{\wedge A} = \left( \begin{array}{ccccccccc}
1 & 1 & 1 & -2 & -1 & -1 & 0 & 0 & 0\\ 
3 & 2 & 3 & -3 & -1 & -3 & 0 & 0 & 0\\ 
3 & 1 & 3 & -3 & -1 & -3 & 0 & 0 & 0\\ 
1 & 1 & 1 & 0 & 0 & 0 & -2 & -1 & -1\\ 
1 & 1 & 1 & 0 & 0 & 0 & -3 & -1 & -3\\ 
1 & 1 & 2 & 0 & 0 & 0 & -3 & -1 & -3\\ 
0 & 0 & 0 & 1 & 1 & 1 & -1 & -1 & -1\\ 
0 & 0 & 0 & 1 & 1 & 1 & -3 & -2 & -3\\ 
0 & 0 & 0 & 1 & 1 & 2 & -3 & -1 & -3
 \end{array} \right)
\]
which is full rank. By \eqref{eqn: flattening lower bound}, we obtain that $\uR(T) \geq \frac{9}{2}$ and therefore $\uR(T) \geq 5$. So $\uR(T) = 5$.

Let $Z = (a_1 + a_2)\otimes (b_1+b_3)\otimes (c_2+c_3)$. We show that $\uR(T-2Z) = \uR(T-Z) = 4$. Since $\uR(Z) = 1$, we have $\uR(T - 2Z) , \uR(T - Z) \geq 4$. From the expression of $T$, it is clear that $\uR(T - 2Z) \leq 4$ so equality holds. Moreover, we observe
\begin{equation}\label{eqn: decomp of T-Z}
\begin{aligned}
  T-Z &= 2( a_1 + a_2 + \textfrac{1}{2} a_3) \otimes ( b_1 + \textfrac{1}{2} b_2 + b_3) \otimes (\textfrac{1}{2} c_1 + c_2 + c_3) + \\
     &+  (a_2 + \textfrac{1}{2} a_3) \otimes b_2 \otimes (\textfrac{1}{2}c_1 + c_2 ) + \\
     &+  (a_1 + \textfrac{1}{2} a_3) \otimes (b_1 + \textfrac{1}{2}b_2) \otimes c_1 + \\
     &+  a_3 \otimes (\textfrac{1}{2} b_2 + b_3 ) \otimes (\textfrac{1}{2}c_1 + c_3), \\
\end{aligned} 
\end{equation}

so $\uR(T - Z) = 4$.

By the argument of \eqref{eqn: construction W drop}, we obtain
\[
 T \otimes T = (T - 2Z)^{\otimes 2} + (T - Z) \otimes 2Z + 2Z \otimes (T - Z),
\]
providing $\uR(T \otimes T) \leq \uR(T - 2Z)^2 + 2 \uR(T-Z)  \leq 4^2 + 4 + 4 = 24 < 5^2 $.
\end{proof}

This proves Theorem \ref{thm: not multiplicative}. We observe that the multiplicativity of the flattening lower bound implies that $\uR( T\otimes T) \geq \left( \frac{9}{2} \right)^2 = 20.25$, providing that $\uR(T \otimes T) \geq 21$.

\begin{remark}
We explain how we determined the decomposition of $T-Z$ in the proof of Proposition \ref{prop: explicit drop for C3 ot C3 ot C3}. After a change of basis on $A,B$ and $C$, we can rewrite $T - Z$ as
\begin{align*}
T' &= a_1 \otimes b_2 \otimes c_3 +  a_3 \otimes b_1 \otimes c_2 + a_2 \otimes b_3 \otimes c_1 +  \\ + &(\textsum_1^3 a_i)\otimes (\textsum_1^3 b_i) \otimes (\textsum_1^3 c_i) + (a_1 + a_2)\otimes (b_1+b_2)\otimes (c_1+c_2).
\end{align*}
Identify $A,B,C$ via the isomorphism $a_i \leftrightarrow b_{i} \leftrightarrow c_{i}$ for $i=1,2,3$. This identification defines a natural action of the symmetric group $\frakS_3$ which permutes the three factors and $T'$ is invariant under the action of the $3$-cycle of $\frakS_3$. We numerically searched for a decomposition of $T'$ that was invariant (as a set) under the action of the subgroup of $\frakS_3$ generated by the $3$-cycle, namely a decomposition of the form
\[
 T' = v_1 \otimes v_2 \otimes v_3 + v_2 \otimes v_3 \otimes v_1 + v_3 \otimes v_1 \otimes v_2 + u \otimes u \otimes u.
\]
Searching for a decomposition of this form rather than a general one allows for a considerable reduction in the dimension of the search space. The set of decompositions of $T'$ having this form determines an affine subvariety in the $12$-dimensional space of $4$-tuples $(v_1,v_2,v_3,u)$ and a numerical algebraic geometry software such as Bertini (see \cite{BatHauSomWam:BertiniSoftware}) can easily determine the \emph{numerical irreducible decomposition} of this variety. It turns out that the variety of decompositions has $25$ irreducible components of dimension $3$ and degree $3$. The expression in \eqref{eqn: decomp of T-Z} is the result of a sampling procedure in which we searched for a decomposition easy to present and to verify by hand. This approach was motivated by Conjecture 4.1.4.2 of \cite{Landsberg:GeoComplTh}; even though this conjecture is false in general (as \cite{Shitov:CounterexampleComon} proved that Comon's conjecture is false), works such as \cite{ChiIkeLanOtt:GeometryRankDecompMatMult}, \cite{BalIkeLanRyd:GeomRankDecompMatMultII} and \cite{Conner:Rank18WaringDecomp} suggest that imposing symmetries is a computationally valid approach, at least in small dimension, to determine explicit decompositions for tensors of low rank.
\end{remark}

We point out that the tensors $T,T-Z$ and $T-2Z$ mentioned in Proposition \ref{prop: explicit drop for C3 ot C3 ot C3} have rank equal to border rank. In particular, the inequalities and the strict submultiplicativity result hold for rank as well.

Moreover, this example generalizes, providing an infinite family of tensors $T_m \in \bbC^3 \otimes \bbC^m \otimes \bbC^m$ with $\uR(T_m) = m+2$ and $\uR(T_m^{\otimes 2}) \leq (m+2)^2 -1$.

\begin{prop}
Let $m \geq 3$ and $A,B,C$ be vector spaces with bases $\{ a_i : i = 1 ,2,3 \}, \{ b_i : i = 1 \vvirg m \}, \{ c_i : i = 1 \vvirg m \}$ respectively. Let
\begin{align*}
T_m := &a_1 \otimes b_1 \otimes c_1 +  a_2 \otimes b_2 \otimes c_2 +a_3 \otimes b_3 \otimes c_3 +  \\ 
+ &(\textsum_1^3 a_i)\otimes (\textsum_1^3 b_i) \otimes (\textsum_1^3 c_i) + 2 (a_1 + a_2)\otimes (b_1+b_3)\otimes (c_2+c_3) + \\ 
+ &a_3 \otimes b_4 \otimes c_4 + \cdots + a_3 \otimes b_m \otimes c_m.
\end{align*}
Then $\uR(T_m) = m+2$ and $\uR(T_m^{\otimes 2}) \leq (m+2)^2 -1$.
\end{prop}
\begin{proof}
We consider the flattening map $(T_m)^{\wedge A}_B : A \otimes B^* \to \Lambda^2 A \otimes C$ analogous to \eqref{eqn: strassen's flattening}. If $Z \in A \otimes B \otimes C$ has rank $1$, then $\rank(Z^{\wedge A} _ B) = 2$.

Let $B' = \langle b_1,b_2,b_3\rangle$ and $B'' = \langle b_4 \vvirg b_m \rangle$ so that $B = B' \oplus B''$ and similarly $C'$ and $C''$; then $(T_m)^{\wedge A}_{B} \vert_{A^* \otimes {B'}^\perp} = (T_3)^{\wedge A}_{B'}$ which has rank $9$ by Proposition \ref{prop: explicit drop for C3 ot C3 ot C3}. In particular $\Lambda^2 A \otimes C' \subseteq \Im ((T_m)^{\wedge A}_{B})$.

Observe that, if $i \in \{ 1, 2\}$ and $k \in \{ 4 \vvirg m\}$, then $(T_m)^{\wedge A} _ B (a_i \otimes b_k) = a_i \wedge a_3 \otimes c_k$. This shows that $\Im ((T_m)^{\wedge A}_{B}) \cap \Lambda^2 A \otimes C''$ contains a subspace of dimension at least $2(m-3)$.

We conclude that
\[
 \rank ((T_m)^{\wedge A}_{B} ) \geq 9 + 2(m-3) = 2m +3;
\]
from \eqref{eqn: flattening lower bound}, we obtain $\uR( T_m ) \geq \frac{2m + 3}{2}$ and therefore $\uR( T_m ) = m+2$.

By applying the same argument as in Proposition \ref{prop: explicit drop for C3 ot C3 ot C3}, we obtain $\uR(T_m^{\otimes 2}) \leq (m+2)^2 -1$.
\end{proof}

\section{The secant multidrop lemma}\label{section: secant multidrop}

In this section, we explain the geometric reason that causes the border rank strict submultiplicativity in Proposition \ref{prop: explicit drop for C3 ot C3 ot C3}. Indeed, it relies on a completely general construction that applies to rank and border rank with respect to any variety, for every tensor power and even in more general settings (see Remark \ref{rmk: general filtration}).

We assume that varieties are irreducible. Moreover, we assume that they are non-degenerate, namely they are not contained in a hyperplane so that their span is the entire ambient space. This is not restrictive, as one can always restrict the ambient spaces to be the span of the variety.

The key result is the following \emph{secant multidrop lemma}:

\begin{lemma}[Secant multidrop]\label{lemma: secant multidrop}
Let $X \subseteq \bbP^{N}$ be an algebraic variety. Fix an integer $r \geq 2$ and let $L$ be a line in $\bbP^{N}$ with the following properties:
 \begin{itemize}
  \dotitem there exists $z \in L \cap X$;
  \dotitem there exist distinct elements $q_0,q_1 \in L \cap \sigma_{r}(X)$ with $q_0,q_1$ distinct from $z$;
  \dotitem $L \not \subseteq \sigma_{r}(X)$.
 \end{itemize}
Then there exists $p$ such that $\uR_{X}(p) = r +1$ and for every $k \geq 2$
\[
\uR_{(X)^{\times k}} (p^{\otimes k}) \leq \frac{1}{2} \biggl((r+1)^k + 2r^k - (r-1)^k \biggr) < (r+1)^k .
\]

In particular $p$ verifies strict submultiplicativity of border rank.
\end{lemma}
\begin{figure}[!htb] 
 \includegraphics[scale=2.2]{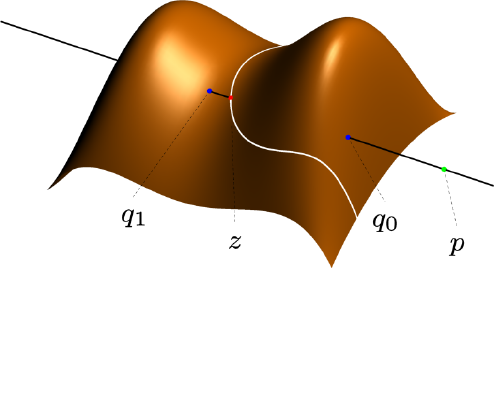}
 \vspace{-2cm}
 \caption{Schematic representation of Lemma \ref{lemma: secant multidrop}: The orange surface represents $\sigma_r(X)$, the white curve represents $X$ and the black line represents $L$.}
\end{figure}

\begin{proof}
Observe that, since $L \not \subseteq \sigma_{r}(X)$, then $L \cap \sigma_{r}(X)$ consists of finitely many points. Notice that $L$ is spanned by $z$ and $q_0$. Working on an affine chart contained in $\bbP^{N} \smallsetminus \{ z \}$, we fix a local parameter $\eps$ so that $L$ is parametrized by $\ell(\eps) = q_0 + \eps z$. Up to rescaling the local parameter assume $\ell(1) = q_1$.

Define the sequence $q_j = \ell(j) = q_0 + jz$. Since $L \cap \sigma_{r}(X)$ is finite, there exists $j_0$ such that $q_{j_0} \notin \sigma_{r}(X)$. Since $j_0 \geq 2$, up to redefining $q_{j_0-2}$ and $q_{j_0 -1}$ to be $q_0$ and $q_1$, we may assume $j_0 = 2$; let $p = q_{2}$. We have $p \in \sigma_{r+1}(X)$ and $q_1 = p - z, q_0 = p-2z \in \sigma_{r}(X)$. 

We verify
\begin{equation}\label{eqn: high drop formula}
p^{\otimes k} = q_0 ^{\otimes k} + 2 \sum_{\substack{ S \subseteq \{ 1 \vvirg k\} \\ \vert S \vert \ \text{odd}}} P_S 
\end{equation}
where $P_S$ is an element of $Seg(\bbP^N \ttimes \bbP^N)$ which is $z$ on the factors corresponding to indices in $S$ and $q_1$ on the factors corresponding to indices not in $S$. In order to prove \eqref{eqn: high drop formula}, we use $q_0 = p-2z$ and $q_1 = p-z$ and expand the right-hand side as a linear combination of tensor products whose factors are $p$ and $z$. These terms are linearly independent because $p$ and $z$ are linearly independent. We claim that all the coefficients on the right-hand side are $0$ except the one of $p^{\otimes k}$ which is $1$.

First, the coefficient of $p^{\otimes k}$ is $1$ because the coefficient of $p^{\otimes k}$ in $q_0^{\otimes k} = (p- 2z)^{\otimes k}$ is $1$ and  $p^{\otimes k}$ does not appear in the summation because $S \neq \emptyset$. Now, fix $m \geq 1$. Without loss of generality, we prove that the coefficient of $z^{\otimes m} \otimes p^{\otimes (k-m)}$ is $0$. The coefficient of this term in $(p-2z)^{\otimes k}$ is $(-1)^m2^m$. The coefficient of this term in $P_S$ is $0$ if $S \not\subseteq \{ 1 \vvirg m\}$ and it is $(-1)^{m - \vert S \vert}$ is $S \subseteq \{ 1\vvirg m\}$. Notice that $(-1)^{m - \vert S \vert} = (-1)^{m-1}$ because $\vert S \vert$ is odd. Since the number of subsets of $\{ 1 \vvirg m\}$ having odd cardinality is $2^{m-1}$, we conclude that the coefficient of $ z^{\otimes m} \otimes p^{\otimes (k-m)}$ on the right-hand side of \eqref{eqn: high drop formula} is $(-1)^m 2^m + 2 \cdot 2^{m-1} (-1)^{m-1} = 0$. This shows that \eqref{eqn: high drop formula} holds.

Observe that $\uR_{(X)^{\times k}}(P_S) \leq r^{k - \vert S \vert}$. Passing to the border rank in \eqref{eqn: high drop formula}, and using subadditivity and submultiplicativity of border rank, we obtain
\begin{align*}
 \uR_{(X)^{\times k}} (p^{\otimes k}) &\leq \uR_{(X)^{\times k}} ( q_0 ^{\otimes k} )  + \sum_{\substack{ S \subseteq \{ 1 \vvirg k\} \\ \vert S \vert  \text{ odd}}} \uR_{(X)^{\times k}}( P_S )  \\
 &\leq r^k + \sum_{\substack{ S \subseteq \{ 1 \vvirg k\} \\ \vert S \vert  \text{ odd}}} r^{k - \vert S \vert} = \\
 &= r^k + \sum_{\substack{j = 1 \vvirg k \\ j \text{ odd}}} \binom{k}{j} r^{k-j} = r^k + \frac{1}{2} ((r+1)^k - (r-1)^k ),
\end{align*}
and from this we conclude.
\end{proof}

\begin{remark}
 We observe that for $k=2$, \eqref{eqn: high drop formula} reduces to \eqref{eqn: construction W drop} with $p$ playing the role of $W$ and $z$ playing the role of $\frac{1}{2} a_1 \otimes b_1 \otimes c_1$. This particular case is the one that has been used in the proof of Proposition \ref{prop: explicit drop for C3 ot C3 ot C3}. Moreover, the different factors play essentially independent roles and the same proof applies to any $k$ varieties $X_i \subseteq \bbP^{N_i}$ (for $i=1 \vvirg k$) such that each of them satisfies the hypothesis of Lemma \ref{lemma: secant multidrop} for some line $L_i$ and points $z^{(i)},q_0^{(i)}$ and $q_1^{(i)}$ in $\bbP^{N_i}$.
\end{remark}

We say that a line $L$ that satisfies the hypothesis of Lemma \ref{lemma: secant multidrop} is a secant with a double drop for $\sigma_{r+1}(X)$. It is in general not clear for which varieties $X$ such a line exists. In the following, we restrict the analysis to the case where $X$ has a secant variety that is a hypersurface. In this case, classical facts about intersection multiplicity allow us to determine sufficient conditions for which a line with a double drop exists.

The multiplicity of a point $p$ in a variety $S$ (see, e.g., \cite{Harris:AlgGeo}, Lecture 20), denoted $\mult_S(p)$, is defined as the degree of the tangent cone to $S$ at $p$, namely $\deg TC_pS$. If $\mult_S(p) = 1$, then $S$ is smooth at $p$ and $TC_p(S) = T_p(S)$ is the tangent space of $S$ at $p$. If $\mult_S(p) = \deg(S)$, then $S$ is a cone over the point $p$ and indeed it coincides with $TC_p(S)$. If $S$ is a hypersurface defined by the polynomial $f$, and $\frakm_p$ is the maximal ideal cutting out the point $p$, then $\mult_p(S) = \max \{ k :  f \in \frakm_p^k \}$, where $\frakm_p^k$ denotes the $k$-th power of the ideal $\frakm_p$. More precisely we have the following:
\begin{remark}\label{rmk: multiplicity in hypersurfaces}
Let $S \subseteq  \bbP^N$ be a hypersurface of degree $d$ and let $f \in \bbC[\bfx]$ be its equation, that is a homogeneous polynomial of degree $d$. Let $p \in S$; up to a change of coordinates, suppose $p = (1,0\vvirg 0)$. Write $f = \sum_{0}^d f_{j} x_0^{d-j}$ where $f_{i}$ are homogeneous of degree $i$ in $x_1 \vvirg x_N$; since $p \in S$, we have $f(p) = f_0 = 0$. Let $m$ be the minimum integer such that $f_m \neq 0$. Thus, the tangent cone of $S$ at $p$ is $TC_p(S) = \{ x \in \bbP^N : f_m = 0\}$ and the multiplicity of $p$ in $S$ is $m$. Notice that in this case $\frakm_p = (x_1 \vvirg x_N)$ and we have $f \in \frakm_p^m$ and $f \notin \frakm_p^{m+1}$.
\end{remark}

We also recall the following basic facts, which are consequence of \cite{Mum:ComplProjVars}, Prop. 5.10:

\begin{remark}\label{rmk: lines and multiplicity}
Let $S$ be a hypersurface of degree $d$ and let $f \in \bbC[\bfx]$ be its equation. Let $L$ be a line in $\bbP^N$, $L \not \subseteq S$. Then $L$ intersects $S$ at $d$ points, counted with multiplicity (as zeros of a univariate polynomial). If $L$ is generic, then the $d$ points are distinct.

Now, fix $p \in S$ with $p=(1,0\vvirg 0)$ as before, and write $f = \sum_{0}^d f_{j} x_0^{d-j}$ as in Remark \ref{rmk: multiplicity in hypersurfaces}. Let $L$ be a line through $p$ in $\bbP^N$; parametrize $L$ locally by a parameter $\eps$ with $L = \{ (1, \lambda_1 \eps \vvirg \lambda _N \eps) : \eps \in \bbC \}$ where $\lambda_1 \vvirg \lambda_N$ are constants. Since $f_i$ is homogeneous of degree $i$ and $x_0 = 1$ on $L$, we have
\[
 f\vert_{L} = \textsum f_{j}( \lambda_1 \vvirg \lambda_N) \eps^j = \eps^m \left(\textsum f_{j}( \lambda_1 \vvirg \lambda_N) \eps^{j-m} \right)
\]
where $m = \mult_S(p)$. This shows that $f\vert_L$ has a zero of multiplicity at least $m$ at $\eps = 0$, corresponding to the point $p$. If $L$ is chosen generically among lines passing through $p$, then $\textsum f_{j}( \lambda_1 \vvirg \lambda_N) \eps^{j-m} $ is a generic univariate polynomial in $\eps$: in particular it has distinct zeros and it does not vanish at $\eps = 0$. We deduce that if $L$ is generic, then the multiplicity of the zero at $\eps = 0$ of the univariate polynomial $ f\vert_{L}$ is exactly $\mult_{S}(p)$ and that is the only zero of $f\vert_{L}$ which is not simple.
\end{remark}

From these remarks, we deduce the following result
\begin{prop}\label{prop: multisecant lemma for hypersurfaces}
 Let $X \subseteq \bbP^N$ be a variety such that $\sigma_r(X)$ is a hypersurface. Let $z \in X$. If $\mult_{\sigma_r(X)}(z) \leq \deg(\sigma_r(X)) -2$, then a generic line $L$ through $z$ is a line with a secant double drop for $\sigma_{r+1}(X)$.
\end{prop}
\begin{proof}
 This follows immediately from Remark \ref{rmk: lines and multiplicity}. A generic line through $z$ intersects $\sigma_r(X)$ in $\deg(\sigma_r(X))$ points counted with multiplicity. By genericity, one zero of multiplicity exactly $\mult_{\sigma_r(X)}(z)$ is at $z$ and there are other $  \deg(\sigma_r(X)) - \mult_{\sigma_r(X)}(z)$ distinct points of intersection on $L$. Since $\mult_{\sigma_r(X)}(z) \leq \deg(\sigma_r(X)) -2$, there are at least two points of intersection between $\sigma_r(X)$ and $L$ other than $z$. Set $q_0$ and $q_1$ to be two of these points. Then they satisfies the hypotheses of Lemma \ref{lemma: secant multidrop}.
\end{proof}

We conclude this section observing the following.
\begin{remark}\label{rmk: general filtration}
We point out that the only properties of $\sigma_r(X)$ that have been used in the proof of Lemma \ref{lemma: secant multidrop} are subadditivity of border rank under sum of tensors and submultiplicativity under tensor product. In particular, the same argument applies to any sequence of varieties $Y_1\subseteq Y_2 \subseteq \cdots$ in $\bbP^{N}$ such that there are varieties $W_1 \subseteq W_2 \subseteq \cdots$ in $\bbP^{(N+1)^k -1}$ with the property that $J(W_{s_1}, W_{s_2}) \subseteq W_{s_1+s_2}$ (subadditivity -- here $J(-,-)$ is the join of varieties in the sense of \cite{Russo:TangentsSecants}), and $Seg(Y_{r_1} \ttimes Y_{r_k}) \subseteq W_{r_1 \cdots r_k}$ (submultiplicativity). Indeed, the analog of \eqref{eqn: high drop formula} is even more general and does not even require that the $Y_i$'s or the $W_j$'s are algebraic varieties.

In particular, Lemma \ref{lemma: secant multidrop} applies to \emph{cactus varieties} (see, e.g., \cite{BuczBucz:SecantVarsHighDegVeroneseReembeddingsCataMatAndGorSchemes}). We observe that indeed, since every zero-dimensional scheme of length $4$ in $\bbP^2 \times \bbP^2 \times \bbP^2$ is smoothable, one obtains that the $4$-th cactus variety of $\bbP^2 \times \bbP^2 \times \bbP^2$ coincides with the $4$-th secant variety; in particular this implies that the cactus border rank of the tensor $T$ in Proposition \ref{prop: explicit drop for C3 ot C3 ot C3} is $5$. Thus, Proposition \ref{prop: explicit drop for C3 ot C3 ot C3} shows that the cactus border rank of $T^{\otimes 2}$ is at most $24$ providing an example of strict submultiplicativity of cactus rank and cactus border rank.
\end{remark}

\section{Some examples of border rank submultiplicativity via lines with drops} \label{section: two hypersurface examples}

In this section, we present three examples where we can apply Lemma \ref{lemma: secant multidrop} and Proposition \ref{prop: multisecant lemma for hypersurfaces}.  The first one is an extensive presentation of the example of Proposition \ref{prop: explicit drop for C3 ot C3 ot C3}. The second and third examples are in the setting of algebraic curves. We show how the phenomenon of the drop along a line occurs for the second secant variety of a normal curve of genus $1$ in $\bbP^4$ and in $\bbP^6$ and for the second secant variety of a normal curve of genus $2$ in $\bbP^4$.

\subsection{The Segre embedding of $\bbP^2 \times \bbP^2 \times \bbP^2$} As in Section \ref{section: P2xP2xP2 basic}, let $A,B,C$ be $3$-dimensional vector spaces and let $X = \bbP A \times \bbP B \times \bbP C$ be the variety of (projective classes of) rank $1$ tensors in $\bbP (A \otimes  B \otimes C)$.

We present a proof of the classical fact that $\sigma_4(X)$ is a hypersurface of degree $9$ (see, e.g., \cite{LanMan:IdealSecantVarsSegre}, Prop. 6.1). Let $\hat{X} \subseteq A \otimes B \otimes C$ be the affine cone over $X$. First, we observe $\codim ( \sigma_4(\hat{X})) \leq 1$.

For every algebraic variety $X \subseteq \bbP^N$, by Terracini's Lemma (see, e.g., \cite[\S5.3]{Lan:TensorBook}), we have
\[
\dim \sigma_r(\hat{X}) = \dim \left( T_{z_1} \hat{X} + \cdots + T_{z_r} \hat{X} \right)
\]
where $z_1 \vvirg z_r$ are generic points of $\hat{X}$ and $T_{z_i} \hat{X}$ denotes the (affine) tangent space to $\hat{X}$ at $z_i$. In our setting, if $a \in A,b\in B, c\in C$, we have $T_{a\otimes b \otimes c} \hat{X} = A \otimes b \otimes c + a \otimes B \otimes c + a \otimes b \otimes C$ where $A \otimes b \otimes c$ is the linear subspace in $A \otimes B \otimes C$ consisting of the elements of the form $a' \otimes b \otimes c$ with $a' \in A$ (and similarly for the other two summands).

Consider the four points of $\hat{X}$ defined by $z_i = a_i \otimes b_i \otimes c_i$ for $i=1,2,3$ and let $z_4 = (\textsum_1^3 a_i)\otimes (\textsum_1^3 b_i) \otimes (\textsum_1^3 c_i)$. One can show via a simple calculation of the rank of a matrix (which in the most naive approach has size $36 \times 27$), that $\dim \bigl( T_{z_1} \hat{X} + \cdots + T_{z_4} \hat{X} \bigr) = 26$ obtaining that $\dim \sigma_4(\hat{X}) \geq 26$. 

On the other hand, it is clear that $\sigma_4(X)$ is a proper subvariety of $\bbP (A\otimes B \otimes C)$, as the tensor $T$ defined in Proposition \ref{prop: explicit drop for C3 ot C3 ot C3} verifies $\uR(T) = 5$. A geometric argument is given in Lemma 3.6 of \cite{AboOttPet:InductionSecantSegre}. We show that the flattening method used in Proposition \ref{prop: explicit drop for C3 ot C3 ot C3} gives indeed the equation of $\sigma_4(X)$. Define 
\[
 Flat_{A \otimes B^*, \Lambda^2 A \otimes C } : A \otimes B \otimes C \to \Hom (A \otimes B^* , \Lambda^2 A \otimes C)
\]
given by $T \mapsto T_B^{\wedge A}$, defined as in Proposition \ref{prop: explicit drop for C3 ot C3 ot C3}. The function $A \otimes B \otimes C \to \bbC$ defined by $\det ( T_B^{\wedge A})$ is a homogeneous polynomial $\scrS \in S^9 ( A^* \otimes B^* \otimes C^*)$ and by \eqref{eqn: flattening lower bound} and the discussion in the proof of Proposition \ref{prop: explicit drop for C3 ot C3 ot C3}, we have that $\sigma_4( X) \subseteq \{ \scrS = 0\}$. To prove that equality holds, it suffices to show that $\scrS$ is irreducible. Let $G = SL(A) \times SL(B) \times SL(C)$; as mentioned in Section \ref{subsection: gvarieties}, $X$ is a $G$-variety, therefore the equation of $\sigma_4(X)$ is a $G$-invariant in $S^d (A^* \otimes B^* \otimes C^*)$ with $d = \deg(\sigma_4(X))$; since $\dim A = \dim B = \dim C = 3$, $G$-invariants in $S^d (A^* \otimes B^* \otimes C^*)$ can only occur in degree $d = 3\delta$ for some $\delta$. It is immediate that indeed $\scrS$ is a $G$-invariant. If it was reducible, by unique factorization, it would have a nontrivial factor generating a $1$-dimensional representation of $G$, that is an invariant as well; in particular, if $\scrS$ was reducible, then $\scrS$ would have a $G$-invariant factor $f \in S^3 (A^* \otimes B^* \otimes C^*)$; the dimension of the space of $G$-invariants in $S^3(A^* \otimes B^* \otimes C^*)$ is the Kronecker coefficient $k_{(1^3),(1^3),(1^3)}$, which is $0$ because $k_{(1^3),(1^3),(1^3)} = k_{(3),(3),(1^3)} = \dim \Hom_{\frakS_3}([3],[3]\otimes [1^3]) = \dim \Hom_{\frakS_3}([3],[1^3]) = 0$ by Schur's Lemma (see, e.g., \cite{FulHar:RepTh}, Lec. 1, Lemma 1.7).

Explicitly, if $T = \sum_1^3 t_{ijk} a_i \otimes b_j \otimes c_k$, the matrix of $T_B^{\wedge A}$ in the bases $\{a_i \otimes \beta_j : i,j = 1 \vvirg 3\}$ of $A \otimes B^*$ and $\{ a_i \wedge a_j \otimes c_k: i,j,k = 1 \vvirg 3 \}$ of $\Lambda^2 A \otimes C$ is 

\begin{equation} \label{eqn: Strassen flattening matrix}
 T_B^{\wedge A} = 
\left( \begin{array}{ccccccccc}
t_{211} & t_{221} & t_{231} & -t_{111} & -t_{121} & -t_{131} & 0 & 0 & 0\\ 
t_{212} & t_{222} & t_{232} & -t_{112} & -t_{122} & -t_{132} & 0 & 0 & 0\\ 
t_{213} & t_{223} & t_{233} & -t_{113} & -t_{123} & -t_{133} & 0 & 0 & 0\\ 
t_{311} & t_{321} & t_{331} & 0 & 0 & 0 & -t_{111} & -t_{121} & -t_{131}\\ 
t_{312} & t_{322} & t_{332} & 0 & 0 & 0 & -t_{112} & -t_{122} & -t_{132}\\ 
t_{313} & t_{323} & t_{333} & 0 & 0 & 0 & -t_{113} & -t_{123} & -t_{133}\\ 
0 & 0 & 0 & t_{311} & t_{321} & t_{331} & -t_{211} & -t_{221} & -t_{231}\\ 
0 & 0 & 0 & t_{312} & t_{322} & t_{332} & -t_{212} & -t_{222} & -t_{232}\\ 
0 & 0 & 0 & t_{313} & t_{323} & t_{333} & -t_{213} & -t_{223} & -t_{233}
\end{array}\right).
\end{equation}

\begin{lemma}
 Let $z \in \bbP A \times \bbP B \times \bbP C$. Then $\mult_{\sigma_4(X)}(z) \leq 7$.
\end{lemma}
\begin{proof}
Let $z = a_1 \otimes b_3 \otimes c_2$ and consider $T_\eps = z + \eps T$ where $T$ is the tensor of Proposition \ref{prop: explicit drop for C3 ot C3 ot C3}. Then the matrix of \eqref{eqn: Strassen flattening matrix} at $T_\eps$ is 
 \[
(T_\eps)_B^{\wedge A} = \left( \begin{array}{ccccccccc}
\eps & \eps & \eps & -2\eps & -\eps & -\eps & 0 & 0 & 0\\ 
3\eps & 2\eps & 3\eps & -3\eps & -\eps & -3\eps-1 & 0 & 0 & 0\\ 
3\eps & \eps & 3\eps & -3\eps & -\eps & -3\eps & 0 & 0 & 0\\ 
\eps & \eps & \eps & 0 & 0 & 0 & -2\eps & -\eps & -\eps\\ 
\eps & \eps & \eps & 0 & 0 & 0 & -3\eps & -\eps & -3\eps-1\\ 
\eps & \eps & 2\eps & 0 & 0 & 0 & -3\eps & -\eps & -3\eps\\ 
0 & 0 & 0 & \eps & \eps & \eps & -\eps & -\eps & -\eps\\ 
0 & 0 & 0 & \eps & \eps & \eps & -3\eps & -2\eps & -3\eps\\ 
0 & 0 & 0 & \eps & \eps & 2\eps & -3\eps & -\eps & -3\eps
\end{array}\right).
 \]
We have
\[
 \det(T_\eps) = -4\eps^9 - 9\eps^8 - 4\eps^7 = -\eps^7 (4\eps^2+9\eps+4) 
\]
which has a zero of multiplicity $7$ at $0$. This shows $\mult_{\sigma_4(X)}(z)\leq 7$. By the action of $G$, $z$ is equivalent to any point of $\bbP A \times \bbP B \times \bbP C$ and this concludes the proof.
\end{proof}

By Proposition \ref{prop: multisecant lemma for hypersurfaces}, we have that a generic line a rank $1$ tensor satisfies the hypotheses of Lemma \ref{lemma: secant multidrop}. The construction of $T$ in Section \ref{section: P2xP2xP2 basic} reflects exactly the general construction of Lemma \ref{lemma: secant multidrop}, by taking $p =T$, $z= Z$, $q_0 = T - 2Z$ and $q_1= T - Z$.

\subsection{The elliptic normal quintic}

In the setting of algebraic curves, the following remark will be useful
\begin{remark}\label{rmk: degree sigma2 of a curve}
 Let $C \subseteq \bbP^N$ be a smooth curve of degree $d$ and genus $g$. Then
 \[
  \deg( \sigma_2(C)) = \frac{(d-1)(d-2)}{2} - g.
 \]
The proof of this fact is classical. Applying a generic projection on $\pi: \bbP^N \dashto \bbP^2$, one reduces to compute the number of nodes of the plane curve $\pi(C)$, of degree $d$ and genus $g$. This number is indeed $\frac{(d-1)(d-2)}{2} - g$. 
 \end{remark}
 
Let $X$ be an elliptic normal curve in $\bbP^4$, which is a curve of degree $5$ and genus $1$. The ideal of $X$ is generated by $5$ quadrics; using coordinates $x_0 \vvirg x_4$ on $\bbP^4$, we have an explicit example given by the five quadrics $q_i = x_i^2 - x_{i+1}x_{i-1} + x_{i+2}x_{i-2}$, where $i = 0 \vvirg 4$ and the indices are to be read $\mathrm{mod} \ 5$ (see, e.g., \cite{Hulek:ProjGeomEllipticCurves}, Ch. IV).

Palatini's Lemma (see, e.g., \cite[Proposition 1.1.2]{Russo:TangentsSecants}) guarantees that if $C \subseteq \bbP^N$ is a curve, then $\dim (\sigma_r(C)) = \min \{ 2r-1, N\}$. In particular $\dim (\sigma_2(X)) = 3$, so $\sigma_2(X)$ is a hypersurface in $\bbP^4$. 

The equation of $\sigma_2(X)$ is the determinant of the Jacobian matrix of the quadrics $q_0 \vvirg q_4$, namely $f = \det(J(\bfx))$ where 
\[
 J = \left[ \begin{array}{ccccc} 
      2 {x}_{0}& {-{x}_{2}}& {x}_{4}& {x}_{1}& {-{x}_{3}}\\
      {-{x}_{4}}& 2 {x}_{1}& {-{x}_{3}}& {x}_{0}& {x}_{2}\\
      {x}_{3}& {-{x}_{0}}& 2 {x}_{2}& {-{x}_{4}}& {x}_{1}\\
      {x}_{2}&{x}_{4}& {-{x}_{1}}& 2 {x}_{3}& {-{x}_{0}}\\
      {-{x}_{1}}& {x}_{3}& {x}_{0}& {-{x}_{2}}& 2 {x}_{4}\\
       \end{array}
 \right].
\]

In particular $\deg (\sigma_2(X)) = 5$ (in accordance with Remark \ref{rmk: degree sigma2 of a curve}) and it turns out that $\mult_{\sigma_2(X)}(z) = 3$ for every $z \in X$ (see, e.g., \cite{Fisher:GenusOnePfaffians}, Lemma 6.7).

By Proposition \ref{prop: multisecant lemma for hypersurfaces}, a generic line through $z$ satisfies the hypotheses of Lemma \ref{lemma: secant multidrop}. 

We construct an explicit example
\begin{example}
We work in the affine space $\bbC^5$. Consider the four points in $\bbC^5$:
 \begin{align*}
  z_0 &= (0,1,-1,1,-1),\\
  z_1 &= (-1,0,1,-1,1),\\
  z_2 &= (1,-1,0,1,-1),\\
  z_3 &= (-1,1,-1,0,1),\\
 \end{align*}
 obtained by cyclic permutation of the coordinates (the fifth point that would arise is not necessary in the construction). We have $z_i \in \hat{X}$ for $i=0\vvirg 3$; so $z_i + z_j \in \sigma_2(\hat{X})$ for $i,j = 0 \vvirg 3$. Let $p = z_1 + z_2 + 2 z_3$. We have $\rank (J(p)) = 5$, so $p \notin \sigma_2(\hat{X})$, and we obtain $\uR_{X} (p) = 3$. However, it is clear that $q_0 = p - 2z_3 \in \sigma_2(\hat{X})$ and one can verify that $q_1 = p - z_3 \in \sigma_2(\hat{X})$, because $J(q_1)$ is singular. We conclude that the line through $p$ and $z$ is a line with a double drop for $X$ and $p$ satisfies strict submultiplicativity of border rank for $r = 2$.
\end{example}

We obtain the following result.
\begin{prop}
 Let $X$ be an elliptic normal quintic in $\bbP^4$. Then the $X$-rank and the $X$-border rank are not multiplicative under tensor product.
\end{prop}

More generally, let $X_m$ be an elliptic normal curve in $\bbP^{2m}$. Then $\sigma_m(X_m)$ is a hypersurface; Conjecture 6.8 in \cite{Fisher:GenusOnePfaffians} states that $\deg(\sigma_m(X_m)) = 2m+1$ and that the equation is a polynomial $F_m$ such that $F_m^m$ is the determinant of the Jacobian matrix of the $m$ generators of $I(\sigma_{m-1}(X_m))$. The conjecture is verified in \cite{Fisher:GenusOnePfaffians} for $m = 3$, where $\sigma_3(X_3)$ is a hypersurface of degree $7$ in $\bbP^6$. In this case, we can verify that $\mult_{\sigma_3(X_3)}(z) = 5$, and therefore Proposition \ref{prop: multisecant lemma for hypersurfaces} can be applied, showing that there are points $p$ in $\bbP^6$ that satisfy strict submultiplicativity of border rank for $r = 3$. In general, if Conjecture 6.8 in \cite{Fisher:GenusOnePfaffians} was true and $\mult_{\sigma_m(X_m)} (z) \leq m-2$ for some $z \in X_m$, then we could guarantee that there are points $p \in \bbP^{2m}$ that satisfy strict submultiplicativity of border rank for $r = m$.

\subsection{A curve of genus $2$ in $\bbP^4$}

We consider the curve described in Remark 3.6 of \cite{Hoff:CurvesGenus2} and we refer to this source for an extensive discussion on the ideal of curves of genus $2$ on rational normal scrolls. Let $X \subseteq \bbP^4$ be the intersection of the cone over a twisted cubic and a generic quadric hypersurface. Explicitly, in the coordinates $x_0 \vvirg x_4$ on $\bbP^4$, consider $X$ given by the equations 
\begin{equation}\label{eqn: equations genus 2}
\begin{aligned}
& x_1^2 - x_0x_2 = 0, \\
& x_2^2 - x_1x_3 = 0, \\
& x_0x_3 - x_1x_2 = 0, \\
& x_0^2 + x_1^2 + x_2^2 + x_3^2 + x_4^2 = 0.
\end{aligned}
\end{equation}

The variety $X$ is a curve of genus $2$ and degree $6$ in $\bbP^4$. From Remark \ref{rmk: degree sigma2 of a curve}, we have $\deg(\sigma_2(X)) = 8$, and indeed the equation of $\sigma_2(X)$ (obtained with a Macaulay2 script in this case -- see \cite{M2}) is 
{\small \begin{dmath*}
F = 4x_0^2x_1^6-12x_0^3x_1^4x_2+4x_0x_1^6x_2+9x_0^4x_1^2x_2^2-14x_0^2x_1^4x_2^2+x_1^6x_2^2+12x_0^3x_1^2x_2^3-8x_0x_1^4x_2^3+10x_0^2x_1^2x_2^4-2x_1^4x_2^4+4x_0x_1^2x_2^5+x_1^2x_2^6+4x_0^4x_1^3x_3+4x_0^2x_1^5x_3-6x_0^5x_1x_2x_3-4x_0^3x_1^3x_2x_3+10x_0x_1^5x_2x_3-4x_0^4x_1x_2^2x_3-20x_0^2x_1^3x_2^2x_3+4x_1^5x_2^2x_3+4x_0^3x_1x_2^3x_3-20x_0x_1^3x_2^3x_3+16x_0^2x_1x_2^4x_3-8x_1^3x_2^4x_3+10x_0x_1x_2^5x_3+4x_1x_2^6x_3+x_0^6x_3^2+2x_0^4x_1^2x_3^2+x_0^2x_1^4x_3^2+4x_0^3x_1^2x_2x_3^2+16x_0x_1^4x_2x_3^2-2x_0^4x_2^2x_3^2-20x_0^2x_1^2x_2^2x_3^2+10x_1^4x_2^2x_3^2-4x_0^3x_2^3x_3^2-20x_0x_1^2x_2^3x_3^2+x_0^2x_2^4x_3^2-14x_1^2x_2^4x_3^2+4x_0x_2^5x_3^2+4x_2^6x_3^2-4x_0^2x_1^3x_3^3+12x_0^3x_1x_2x_3^3+4x_0x_1^3x_2x_3^3+4x_0^2x_1x_2^2x_3^3+12x_1^3x_2^2x_3^3-4x_0x_1x_2^3x_3^3-12x_1x_2^4x_3^3-2x_0^4x_3^4-2x_0^2x_1^2x_3^4-4x_0x_1^2x_2x_3^4+2x_0^2x_2^2x_3^4+9x_1^2x_2^2x_3^4+4x_0x_2^3x_3^4-6x_0x_1x_2x_3^5+x_0^2x_3^6+4x_1^6x_4^2-12x_0x_1^4x_2x_4^2+6x_0^2x_1^2x_2^2x_4^2-2x_1^4x_2^2x_4^2+4x_0^3x_2^3x_4^2+4x_0^2x_2^4x_4^2-2x_1^2x_2^4x_4^2+4x_0x_2^5x_4^2+4x_2^6x_4^2+8x_0^2x_1^3x_3x_4^2+4x_1^5x_3x_4^2-12x_0^3x_1x_2x_3x_4^2-4x_0x_1^3x_2x_3x_4^2-4x_0^2x_1x_2^2x_3x_4^2-4x_0x_1x_2^3x_3x_4^2-12x_1x_2^4x_3x_4^2+2x_0^4x_3^2x_4^2+2x_0^2x_1^2x_3^2x_4^2+4x_1^4x_3^2x_4^2-4x_0x_1^2x_2x_3^2x_4^2+2x_0^2x_2^2x_3^2x_4^2+6x_1^2x_2^2x_3^2x_4^2+8x_0x_2^3x_3^2x_4^2+4x_1^3x_3^3x_4^2-12x_0x_1x_2x_3^3x_4^2+2x_0^2x_3^4x_4^2-3x_1^2x_2^2x_4^4+4x_0x_2^3x_4^4+4x_1^3x_3x_4^4-6x_0x_1x_2x_3x_4^4+x_0^2x_3^2x_4^4 
\end{dmath*}
}

Notice that $z = (1,1,1,1,2i) \in X$. We can observe that a generic line through $z$ intersects $\sigma_2(X)$ at $z$ with multiplicity $4$ (for instance the line spanned by $z$ and $(2,1,1,0,0)$). This implies $\mult_{\sigma_2(X)}(z) \leq 4$ and therefore Lemma \ref{lemma: secant multidrop} applies to a generic line through $X$, providing 

\begin{prop}\label{prop: border rank submult for genus 2}
 Let $X \subseteq \bbP^4$ be the curve with equations \eqref{eqn: equations genus 2}. Then the $X$-border rank is not multiplicative under tensor product.
\end{prop}

In fact, Proposition \ref{prop: border rank submult for genus 2} holds for $X$-rank, as well. This will be a consequence of Corollary \ref{corol: rank vs border rank asymptotically} below: indeed, the result of Proposition \ref{prop: border rank submult for genus 2} implies that $\aR_X(p) < \uR_X(p)$ for some $p \in \bbP^4$, and therefore $\aR_X(p) < \rmR_X(p)$, providing $\rmR(p^{\otimes k}) < \rmR(p)^k$ for some $k$.

\section{The tensor asymptotic rank}\label{section:AsymptoticRank}

This section deals with the following notion of asymptotic rank:
\begin{defin}
 Let $X \subseteq \bbP^N$ be an algebraic variety and let $p \in \bbP^N$. The \emph{tensor asymptotic $X$-rank} of $p$ is 
 \begin{equation}\label{eqn: def of asymptotic tensor rank}
\aR_X(p) =  \lim_{k \to \infty} \bigl[ \rmR_{X \ttimes X} (p^{\otimes k}) \bigr]^{1/k}.
 \end{equation}
\end{defin}

We observe that the limit of the sequence $\{ \bigl[\rmR_{X \ttimes X} (p^{\otimes k}) \bigr]^{1/k} : k \in \bbN\}$ exists. This is a consequence of Fekete's Lemma (see, e.g., \cite{PolSze:ProblemsTheoremsAnalysisI}, p. 189): let $r_k = \log( \rmR_{X \ttimes X} (p^{\otimes k}))$; by submultiplicativity, we have $r_{k + \ell} \leq r_k + r_\ell$ and therefore Fekete's Lemma guarantees that the sequence $\frac{1}{k} r_k$ converges, and passing to the exponentials, we conclude that the limit in \eqref{eqn: def of asymptotic tensor rank} exists and is finite.

Our first goal is to show that the tensor asymptotic rank is the same if we consider border rank rather than rank in the definition. When $X$ is the Segre variety of rank $1$ tensors, this result is a consequence of Theorem 8 in \cite{ChrJenZui:NonMultiplicativityTensorRank} and the same interpolation argument applies whenever $X$ is a rational variety. We can apply the same idea in general, but there are varieties for which it is not possible to guarantee that the approximating curve is rational, which is necessary to write a rational parametrization and use the interpolation argument of \cite{ChrJenZui:NonMultiplicativityTensorRank}. Therefore, we use some basic ideas from intersection theory, for which we refer to Ch. 1 and Ch. 2 in \cite{EisHar:3264}. 

\begin{prop}\label{prop: asymptotic rank less than border rank}
 Let $X \subseteq \bbP^N$ be an algebraic variety and let $p \in \bbP^N$. We have $\aR_X(p) \leq \uR_X(p)$.
\end{prop}
\begin{proof}
The argument in this proof is based on the fact that there exists a constant $e$ such that, for every $k$, $\rmR_{X\ttimes X}(p^{\otimes k}) \leq \uR_X(p)^k (ek+1)$. By taking the limit in $k$, we will conclude $\aR_X(p) =  \lim_{k \to \infty} \bigl[ \rmR_{X \ttimes X} (p^{\otimes k}) \bigr]^{1/k} \leq \lim_{k \to \infty} \bigl[ \uR_X(p) ^k (ek+1) \bigr]^{1/k} = \uR_X(p)$.

Suppose $ \uR_X(p) = r$, so that $p \in \sigma_r(X)$. Write $\sigma_r^\circ(X) = \{ q \in \sigma_r(X): \uR(q) = \rmR(q) = r\}$. Then $\sigma_r^\circ(X)$ is Zariski-open in $\sigma_r(X)$. Let $c = \codim(\sigma_r(X))$ and let $L$ be a generic $(c+1)$-dimensional linear subspace of $\bbP^N$ passing through $p$. Then $C = L \cap \sigma_r(X)$ is a (possibly reducible) algebraic curve; let $E$ be an irreducible component of $C$ passing through $p$ with the property that $E \cap \sigma_r^\circ(X) \neq \emptyset$ (this exists by genericity of $L$). Let $E^\circ = E \cap \sigma^\circ_r(X)$, that is Zariski-open in $E$, therefore, since $E$ is a curve, $E \setminus E^\circ$ is finite.  Let $e = \deg(E)$. Then $e+1$ generic points on $E$ span $\langle E \rangle$ (see, e.g., \cite{Shaf:BasicAlgGeom1}, Thm. 3.9). In particular, there exist $e+1$ points on $E^\circ$ such that $p$ is contained in their span; each of these points has rank $r$, so we conclude $R_X(p) \leq r(e+1)$. In this argument, the constant $e$ plays the same role as the error degree in the proof of Thm. 8 in \cite{ChrJenZui:NonMultiplicativityTensorRank}; in fact, if $E$ is a rational curve, $e$ coincides with the error degree.
 
Let $h \in A(\bbP^N)$ be the hyperplane class in the Chow ring of $\bbP^N$. We have $[E] = e h^{N-1}$ (see, e.g., Thm. 2.1 in \cite{EisHar:3264}). Consider the image of the curve $E$ under the diagonal embedding $\Delta: \bbP^N \to \bbP^N \ttimes \bbP^N$ followed by the Segre embedding $Seg : \bbP^N \ttimes \bbP^N \to \bbP^{ (N+1)^k -1}$. Denote by $H \in A^1 ( \bbP ^{(N+1)^k -1})$ the hyperplane class in $\bbP^{(N+1)^k -1}$ and by $h_j \in A^1 ( \bbP^N \ttimes \bbP^N)$ the class of the divisor cut out by a linear equation on the $j$-th factor.

We want to show that $\deg(Seg(\Delta(E)) = ek$. Since both $\Delta$ and $Seg$ are embeddings, and since $E$ is a curve, we have $\deg( Seg(\Delta(E)) = [Seg(\Delta(E))] \cdot H = Seg_* ( \Delta_*([E])) \cdot H$ (where we used the definition of pushforward of cycles via embeddings, as in \cite{EisHar:3264}, Definition 1.19). We determine this number via the push-pull formula (Theorem 1.23 in \cite{EisHar:3264}). Direct calculation provides $Seg^*(H) = h_1 + \cdots + h_k$ and $\Delta^*(h_j) = h$.  The push-pull formula (applied twice) provides
\begin{align*}
Seg_*(\Delta_*([E])) \cdot H = Seg_*(\Delta_*([E]) \cdot Seg^*(H)) = Seg_*(\Delta_*([E] \cdot \Delta^*(Seg^*(H)))).
\end{align*}
Since $\Delta$ and $Seg$ are both embeddings, we have $Seg_*(\Delta_*([E])) \cdot H  = [E] \cdot \Delta^*(Seg^*(H))$ (by identifying the component of top degree in $A(\bbP^N)$ and in $A(\bbP^{(N+1)^k-1})$ with $\bbZ$). We deduce
\[
 \deg(Seg(\Delta(E))  = [E] \cdot \Delta^*(Seg^*(H)) = e h^{N-1} \cdot \Delta^*( h_1 + \cdots + h_k) = e  h^{N-1} \cdot (kh) = ek.
\]
Therefore $ \deg( Seg(\Delta(E)) = ek$.

Now, $Seg(\Delta(E))$ is a curve of degree $ek$, with $ p^{\otimes k} \in Seg(\Delta(E)) \subseteq \sigma_r(X) ^k \subseteq \sigma_{r^k}(X)$ and $Seg(\Delta(E)) \cap \sigma^{\circ}_{r^k(X)} \supseteq Seg(\Delta(E^\circ)) \neq \emptyset$. In particular, $ek+1$ generic points in $Seg(\Delta(E^\circ)) \subseteq \sigma^{\circ}_{r^k}(X) $ span $p^{\otimes k}$ and we obtain 
\[
 \rmR_{X \ttimes X} (p^{\otimes k}) \leq r^k (ek+1) .
\]
\end{proof}

\begin{corol}\label{corol: rank vs border rank asymptotically}
  Let $X \subseteq \bbP^N$ be an algebraic variety and let $p \in \bbP^N$. Then $\aR_X(p) =  \lim_{k \to \infty} \bigl[ \uR_{X \ttimes X} (p^{\otimes k}) \bigr]^{1/k}$.
\end{corol}
\begin{proof}
Define temporarily $\uaR_X(p) =  \lim_{k \to \infty} \bigl[ \uR_{X \ttimes X} (p^{\otimes k}) \bigr]^{1/k}$. We are going to show $\uaR_X(p) = \aR_X(p)$. Since $\uR_X(p) \leq \rmR_X(p)$, we have $\uaR_X(p) \leq \aR_X(p)$.

Fix $\ell$ and apply Proposition \ref{prop: asymptotic rank less than border rank} to $p^{\otimes \ell}$. We have $\aR_{X \ttimes X}(p^{\otimes \ell}) \leq \uR_{X \ttimes X}(p^{\otimes \ell})$. On the other hand, $\aR_{X \ttimes X}(p^{\otimes \ell})^{1/\ell} = \aR_{X}(p)$ directly from the definition: indeed
\[
 \aR_{(X)^{\times \ell}}(p^{\otimes \ell})^{1/\ell} = \left[\lim_{k \to \infty} \bigl[ \uR_{(X)^{\times \ell k}} (p^{\otimes \ell k} ) \bigr]^{1/k}\right]^{1/\ell} = \lim_{k \to \infty} \bigl[ \uR_{(X)^{\times \ell k}} (p^{\otimes \ell k} ) \bigr]^{1/(\ell k)} = \aR_{X}(p)
 \]
as we are just considering the limit of \eqref{eqn: def of asymptotic tensor rank} on a subsequence.

Therefore, for every $\ell$, $\aR_{X}(p) \leq \aR_{X \ttimes X}(p^{\otimes \ell})^{1/\ell} \leq \uR_{X \ttimes X}(p^{\otimes \ell})^{1/\ell}$. Passing to the limit in $\ell$ (which exists by submultiplicativity and Fekete's Lemma), we conclude $\aR_{X}(p) \leq\uaR_X(p)$.
\end{proof}

A completely general consequence of multiplicativity of flattening lower bounds (Section \ref{subsec: flattenings}) is that, for every variety $X \subseteq \bbP^N$ and every point $p \in \bbP^N$, $\aR_X(p) \geq \rmR_X^{Flat}(p)$ where $\rmR_X^{Flat}(p)$ is the best lower bound on $\uR_X(p)$ that can be obtained via flattening methods. In the tensor setting, flattening methods can provide lower bounds beyond the local dimension, and this guarantees the existence of tensors whose tensor asymptotic rank is strictly higher than the local dimension. We point out that, whenever $\uR_X(p)$ is realized by an \emph{integer} flattening lower bound, then $\uR_X(p) = \aR_X(p)$. An explicit example of this phenomenon is provided in the next result.

\begin{prop}\label{prop: multiplicativity for m plus 1}
Let $m\geq 3$ and let $T \in \bbC^m \otimes \bbC^m \otimes \bbC^m$ be a generic tensor of rank $m+1$. Then $\uR(T^{\otimes k} ) = (\uR(T))^k$ for all $k \geq 1$. In particular $\aR (T) = m+1$. 
\end{prop}
\begin{proof}
We write $T \in A \otimes B \otimes C$ with $\dim A = \dim B = \dim C = m$ and we fix bases of $\{ a_i \}, \{b_i\}, \{c_i\}$ of $A,B$ and $C$ respectively as in Section \ref{section: P2xP2xP2 basic} with corresponding dual bases $\{\alpha_i\},\{\beta_i\}$ and $\{\gamma_i\}$. Since border rank is upper semicontinuous and it suffices to prove the statement for 
\[
T = a_1 \otimes b_1 \otimes c_1 + \cdots + a_{m} \otimes b_{m} \otimes c_{m} + (\textsum a_i)\otimes ( \textsum b_i) \otimes (\textsum c_i). 
\]
We consider the first Koszul flattening of $T$ as in \eqref{eqn: strassen's flattening}, namely the standard flattening augmented with the identity on a copy of $A$, followed by the projection of the factor $A \otimes A$ on the skew-symmetric component $\Lambda^2 A$:
\begin{align*}
 T^{\wedge A}_{B} : A \otimes B^* \to \Lambda^2 A \otimes C.
\end{align*}
If $Z$ is a rank $1$ tensor, then $\rank(Z^{\wedge A}_B) = m-1$, so the Koszul flattening provides the lower bound $\uR (T) \geq \dfrac{\rank (T^{\wedge A}_B )}{m-1}$. We will show that $\rank (T^{\wedge A}_B) = m^2 -1 = (m+1)(m-1)$. This gives the multiplicative lower bound $\uR(T) \geq \frac{m^2-1}{m-1} = m+1$ by \eqref{eqn: multiplicative flattening lower bounds}.

For every $i,j$, we have  
\[
 T^{\wedge A}_B ( a_i \otimes \beta_j) = a_i \wedge a_j \otimes c_j + a_i \wedge (\textsum a_\ell) \otimes (\textsum c_\ell), 
\]
where the first summand is $0$ if $i=j$.

Observe that for every $i\neq j$, we have $a_i \wedge a_j \otimes c_j \in \Im T^{\wedge A}_B$. Indeed, we have
\begin{align*}
 T^{\wedge A}_B &( a_i \otimes (\beta_j - \beta_i)) = T^{\wedge A}_B ( a_i  \otimes \beta_j) - T^{\wedge A}_B ( a_i  \otimes \beta_i) \\ 
 &= \left[a_i \wedge a_j \otimes c_j + a_i \wedge \left( \textsum_1^m a_\ell\right) \otimes \left( \textsum_1^m c_\ell\right)\right] - \left[a_i \wedge \left( \textsum_1^m a_\ell\right) \otimes \left( \textsum_1^m c_\ell\right) \right]= \\ 
 & =a_i \wedge a_j \otimes c_j.
\end{align*}
In particular, we have $\langle a_i \wedge a_j \otimes c_j : i,j = 1 \vvirg m \rangle \subseteq \Im T_B^{\wedge A}$, showing $\rank( T_B^{\wedge A} ) \geq 2 \binom{m}{2} = m^2 - m$.

Passing to the quotient modulo $\langle a_i \wedge a_j \otimes c_j : i,j = 1 \vvirg m \rangle$, for every $k = 1 \vvirg m$, we have 
\[
 T^{\wedge A}_B(a_j \otimes \beta_k) \equiv a_j \wedge (\textsum a_\ell) \otimes (\textsum c_\ell).
\]
Observe that these span an $(m-1)$-dimensional space modulo $\langle a_i \wedge a_j \otimes c_j : i,j = 1 \vvirg m \rangle$. We obtain $\rank (T_B^{\wedge A}) \geq m-1 + m^2 - m = m^2 - 1$ and this concludes the proof.
\end{proof}

We conclude this section showing that in the setting of Lemma \ref{lemma: secant multidrop}, if $\uR_X(p) = r+1$, we obtain upper bounds for $\aR_X(p)$ that are strictly smaller than $r+1$.

\begin{remark}
 Let $X \subseteq \bbP^N$ be an algebraic variety and let $z \in X$ and $p \in \bbP^N$ with $\uR_X(p) = r+1$ and $\uR_X(p - z) = \uR_X(p-2z) = r$. Then from Lemma \ref{lemma: secant multidrop}, we have $\uR_X(p^{\otimes k}) \leq \frac{1}{2} ( (r+1)^k + 2r^k - (r-1)^k)$ from which, for every $k \geq 2$, we have
 \begin{equation}\label{eqn: bound on asymptotic}
  \aR_X(p) \leq \left[ \frac{1}{2} ( (r+1)^k + 2r^k - (r-1)^k) \right]^{1/k}.
 \end{equation}
 Let $B(r,k)$ be the right-hand side of \eqref{eqn: bound on asymptotic}. For fixed $r$, $B(r,k)$ is eventually increasing as a function of $k$ and converges to $r+1$. In particular, there is a value $\kappa_r$ for which $B(r,\kappa_r) = \min_{k} (B(r,k))$, which realizes the best possible bound for $ \aR_X(p)$ with the only use of Lemma \ref{lemma: secant multidrop}. We expect $\kappa_r$ to be an increasing function of $r$.
 
The following table records the values of $B(r,k)$ for different values of $k$ when $r = 4$; this is the setting of the tensor $T$ in Proposition \ref{prop: explicit drop for C3 ot C3 ot C3}:
\[
 \begin{array}{c|c}
k & B(4,k) \\ \midrule
1 & 5 \\
2 & 4.898979486 \\
3 & 4.834588127 \\
4 & 4.793563454 \\
5 & 4.768297954\\
6 & 4.754002287 \\
7 & 4.747451133 \\ 
8 & 4.746368884 \\ 
9 & 4.749102849 \\
10 & 4.754435059  
 \end{array}
\]
This shows $4.5 \leq \rmR^{Flat}(T) \leq \aR(T) \leq 4.746368884$ where $T$ is the tensor of Proposition \ref{prop: explicit drop for C3 ot C3 ot C3}.

More generally, for every $p \in \bbP^N$ that arises from Lemma \ref{lemma: secant multidrop}, we have $\rmR_X^{Flat}(p) \leq \aR_X(p) \leq B(r,\kappa_r)$. We record in the following table the values of $\kappa_r$ and the corresponding bound $B(r,\kappa_r)$ for $r = 1 \vvirg 50$ (the decimal expansions have $9$ significant digits). We observe that the value $\kappa_r$ matches the sequence \texttt{A186326} in \cite{OEIS} (this has been checked for $r \leq 100$).
\[
  \begin{array}{c||c|c}
r   & \kappa_r & B(r,\kappa_r) \\ \midrule
1 & 3 & 1.7099759467 \\
2 & 5 & 2.7348800685 \\ 
3 & 6 & 3.7418846152 \\ 
4 & 8 & 4.746368884 \\ 
5 & 9 & 5.7490740939 \\ 
6 & 11 & 6.7507695302 \\ 
7 & 12 & 7.7522561776 \\ 
8 & 14 & 8.7530862563 \\ 
9 & 16 & 9.7539245075 \\ 
10 & 17 & 10.7545150388 \\ 
\end{array}\]
\end{remark}

\section{Other varieties and numerical results}\label{section: other varieties}

There are a number of varieties to which one can potentially apply Lemma \ref{lemma: secant multidrop} and Proposition \ref{prop: multisecant lemma for hypersurfaces}. Every potential example presents two challenges: it is in general hard to determine the degree of the hypersurface, and of course even harder to determine its equation; even in cases where it is possible to give some information about the equation, determining the multiplicity of a point of rank $1$ inside the hypersurface might be difficult. Moreover, there are cases where the secant multidrop argument cannot be applied, in the sense that, in the notation of Proposition \ref{prop: multisecant lemma for hypersurfaces}, $\mult_z(S) = \deg(S) -1$.

This section is dedicated to these additional examples, where either the argument cannot be applied, or we are not able to provide enough information to apply it. In this analysis, we provide some experimental data obtained via numerical algebraic geometry: in particular, degrees of varieties and multiplicities of singular points can be computed numerically using methods based on monodromy (see, e.g., \cite[Ch. 10]{BaHaSoWa:BertiniBook}); in our case, we used the software Bertini \cite{BatHauSomWam:BertiniSoftware} with an algorithm based on the one used in \cite{OedSam:5thSecant5P1s}, enhanced with trace test (see \cite{LeyRodSot:TraceTest}).

\subsection{Variety of singular matrices}

Let $X$ be the variety of rank $1$ $n\times n$ matrices in $\bbP Mat_n$, where $Mat_n$ is the space of $n\times n$ matrices with complex coefficients. We use coordinates $x_{ij}$ on $Mat_n$, corresponding to the entries of a the matrix. The secant variety $\sigma_r(X)$ consists of matrices of rank at most $r$. In particular $\sigma_{n-1}(X)$ is a hypersurface of degree $n$, cut out by the determinant polynomial; it is called the \emph{determinantal hypersurface} in $Mat_n$. Denote $S = \sigma_{n-1}(X)$. It is a standard exercise to show that the multiplicity of a rank $1$ matrix in the determinantal hypersurface is $n-1$ (see \cite[Ch.II]{ArCoGrHa:Vol1} for details). This shows that Lemma \ref{lemma: secant multidrop} cannot be applied to the variety $S$. Indeed, it is classically known that in this case border rank coincides with matrix rank and in particular it is equal to the standard flattening lower bound as explained in Section \ref{subsec: flattenings}, because the identity map on $Mat_n$ defines a flattening map.

The same argument applies whenever a secant variety of $X$ is a hypersurface with a determinantal expression $\det(A(\bfx))$ such that $\rank(A(z)) =1$ for every $z \in X$. In this case, the denominator in Equation \eqref{eqn: flattening lower bound} is $r_0 = 1$, and the flattening lower bound is attained by a generic point. Border rank multiplicativity follows by multiplicativity of flattening lower bounds. Indeed, in this case the tensor asymptotic rank and the border rank coincide. We give some examples to illustrate this phenomenon.

\begin{itemize}
\dotitem the $\left(\frac{d}{2}\right)$-th secant variety of the rational normal curve in $\bbP^d$ for $d = 2\delta$ even: $S = \sigma_\delta(\nu_{2\delta}(\bbP^1)) \subseteq \bbP S^d \bbC^2$: we have $\deg(S) = \delta+1$ and its equation is the determinant of the \emph{catalecticant} flattening $f \mapsto f_{\delta,\delta}$ where $f \in S^d \bbC^2$ defines the map $f_{\delta,\delta} : S^\delta \bbC^{2*} \to S^\delta \bbC^2$ between two spaces of dimension $\delta+1$.

 \dotitem the $9$-th secant variety of the $6$-Veronese embedding of $\bbP^2$: $S = \sigma_9(\nu_6(\bbP^2)) \subseteq \bbP (S^6 \bbC^3)$: we have $\deg (S) = 10$ and its equation is the determinant of the \emph{catalecticant} flattening $f \mapsto f_{3,3}$ where $f \in S^6\bbC^3$ defines the map $f_{3,3} : S^3 \bbC^{3*} \to S^3 \bbC^3$ between two spaces of dimension $10$.
 
 \dotitem the $5$-th secant variety of the $(3,3)$-Segre-Veronese embedding of $\bbP^1 \times \bbP^1$: $S = \sigma_5( \nu_{3,3}(\bbP^1 \times \bbP^1)) \subseteq \bbP (S^3 \bbC^2 \otimes S^3 \bbC^2)$: we have $\deg(S) =6$ and its equation is the determinant of the flattening $T \mapsto T_{(1,2),(2,1)}$ where $T \in S^3 \bbC^2 \otimes S^3 \bbC^2$ defines the map $T_{(1,2),(2,1)} : \bbC^{2*} \otimes S^2 \bbC^{2*} \mapsto S^2 \bbC^{2} \otimes \bbC^{2}$ between two spaces of dimension $6$.
\end{itemize}

A similar case is $\sigma_3(\nu_3(\bbP^2))$, which is a hypersurface of degree $4$ in $\bbP( S^3 \bbC^3)$; its equation is classically known as Aronhold invariant and it arises as any of the nine Pfaffians of size $8$ of Koszul-type flattening $\bbC^3 \otimes \bbC^3 \xto{\id_3 \otimes f_{1,2} } \bbC^3 \otimes S^2 \bbC^3 \xto {\rho_{1,2}} \Lambda^2 \bbC^3 \otimes \bbC^3$ where $f_{1,2}$ is a catalecticant flattening and $\rho_{p,q}$ is the (transpose of the) classical Koszul map, namely $\rho_{p,q}: \Lambda^p U \otimes S^q U \to \Lambda^{p+1} U \otimes S^{q-1} U$ for a vector space $U$ (see, e.g., \cite{Eis:CommutativeAlgebra}, Ch. 17). It is immediate to verify that $\mult_{\sigma_3(\nu_3(\bbP^2))} (z) = 3$ for every $z \in \nu_3(\bbP^2)$ and therefore Lemma \ref{lemma: secant multidrop} cannot be applied.

Many more examples can be constructed in a similar way.

\subsection{Other secant hypersurfaces}

In this last section, we discuss some examples that, in our opinion, might be interesting to investigate further. In particular, these are examples of varieties $X \subseteq \bbP^N$ with the property that $\sigma_r(X)$ is a hypersurface for some $r$. In some cases, we know the degree of the hypersurface $\sigma_r(X)$, either for theoretical reasons or via numerical algebraic geometry. In these cases we are not able to determine the equation of $\sigma_r(X)$ or the value of $\mult_{\sigma_r(X)}(z)$ for $z \in X$, and consequently we are not able to apply Proposition \ref{prop: multisecant lemma for hypersurfaces}. However, we believe that these cases can lead to further examples of strict submultiplicativity, as in the cases presented in Section \ref{section: two hypersurface examples}.

Let $X \subseteq \bbP^N$ be an algebraic variety. There is a standard parameter count that gives an upper bound for the dimension of $\sigma_r(X)$, that is $\min\{ N,  r(\dim (X) + 1) -1\}$; this upper bound is called the expected dimension of $\sigma_r(X)$. In the following examples, all secant varieties of the variety $X$ have the expected dimension.

\begin{enumerate}[(i)]
\item Let $X \subseteq \bbP^{2m}$ be a nondegenerate curve. Then $\sigma_m(X)$ is a hypersurface in $\bbP^{2m}$. In Section \ref{section: two hypersurface examples}, we saw some cases for which Proposition \ref{prop: multisecant lemma for hypersurfaces} provided examples of strict submultiplicativity of $X$-border rank. 

 \item Let $X = \nu_{3k}(\bbP^2) \subseteq \bbP S^{3k}(\bbC^3)$ and $r = 3 \binom{k}{2}$; then $\sigma_r(X)$ is a hypersurface in $\bbP S^{3k}(\bbC^3)$. For $k = 1,2$, we have $\mult_{\sigma_r(X)} (z) = \deg( \sigma_r(X) ) - 1$ for every $z \in X$, as explained in the previous section, so Proposition \ref{prop: multisecant lemma for hypersurfaces} cannot be applied. For $k=3$, we obtain numerically $\deg(\sigma_{18}(X)) = 1292$ and we have numerical evidence suggesting that $\mult_{\sigma_{18}(X)}(z) =1215$ for $z \in X$, so Lemma \ref{lemma: secant multidrop} provides examples of strict submultiplicativity of border rank with $\uR_X(p) = 19$. We expect that for $k \geq 3$, a generic line through $X$ has many points of rank $r$, providing several examples of strict submultiplicativity.
 
 \item Let $X = Seg( (\bbP^1) ^{\times a})$ with $2^a -1$ divisible by $a+1$ (namely $a+1$ is a $2$-Fermat pseudoprime) and let $r = \frac{2^a -1}{a+1}$ ($a \neq 4$); then $\sigma_r(X)$ is a hypersurface in $\bbP ( (\bbC^{2})^{\otimes a})$. If $a = 2$, then $r = 1$ and $X$ is the quadric $\bbP^1 \times \bbP^1$ in $\bbP^3$. When $a = 6$, then $r = 9$ and $\sigma_9((\bbP^1) ^{\times 6})$ is a hypersurface in $\bbP^{63}$. We have numerical evidence that $\deg( \sigma_9((\bbP^1) ^{\times 6}) ) \geq 1.4 \cdot 10^5$ and surprisingly it seems from the numerical calculations that $\mult_{z} (\sigma_9((\bbP^1) ^{\times 6}) ) = \deg( \sigma_9((\bbP^1) ^{\times 6})) -1$, so that it might not be possible to apply Proposition \ref{prop: multisecant lemma for hypersurfaces} to this class of examples.
 
 \item Let $X = \nu_{3,3a}(\bbP^1 \times \bbP^1) \subseteq \bbP( S^{3} \bbC^2 \otimes S^{3a} \bbC^2)$ and let $r = 4a + 1$; then $\sigma_r(X)$ is a hypersurface. For $a=1$, we have $\mult_{\sigma_5(X)} (z) = \deg( \sigma_5(X) ) - 1$ for every $z \in X$ (as explained in the previous section) so Proposition \ref{prop: multisecant lemma for hypersurfaces} cannot be applied. For $(a,b) = (1,2)$, we obtain numerically $\deg(\sigma_{9}(X) )= 28$ and we have numerical evidence suggesting that $\mult_{\sigma_9(X)}(z) = 25$ for $z \in X$, so Lemma \ref{lemma: secant multidrop} provides examples of strict submultiplicative of border rank with $\uR_X(p) = 10$. For $(a,b) = (1,3)$, we obtain numerically $\deg(\sigma_{13}(X)) = 102$ and we have numerical evidence suggesting that $\mult_{z}(\sigma_{13}(X)) = 94$ for $z \in X$, so Lemma \ref{lemma: secant multidrop} provides examples of strict submultiplicativity of border rank with $\uR_X(p) = 14$. We expect that a similar situation occurs for higher values of $a$ providing several cases in which Lemma \ref{lemma: secant multidrop} can be applied.
 \end{enumerate}

Many more examples are available and can be approached with similar techniques.

\bibliographystyle{amsalpha}
\bibliography{bibfile}

 \end{document}